\documentclass[Times1COL]{WileyNJDv5}
\usepackage{amssymb}
\usepackage{mathptmx}
\usepackage{helvet}
\usepackage{courier}
\usepackage{makeidx}
\usepackage{amsmath,amsfonts}
\usepackage{graphicx}
\usepackage{cite}
\usepackage[bottom]{footmisc}
\usepackage{bm}
\usepackage{color}

\begin{document}

\articletype{RESEARCH ARTICLE}

\received{Date Month Year}
\revised{Date Month Year}
\accepted{Date Month Year}
\journal{ZAMM-Zeitschrift fur Angewandte Mathematik und Mechanik}
\volume{00}
\copyyear{2025}
\startpage{1}

\raggedbottom

\title{Closed Form Relations and Higher-Order Approximations of First and Second Derivatives of the Tangent Operator on SE(3)}
\author[1]{Andreas M\"uller}

\titlemark{Closed Form Relations and higher-Order Approximations of First and Second Derivatives of the Tangent Operator on SE(3)}
\authormark{A. M{\"u}ller}

\address[1]{\orgdiv{Institute of Robotics}, \orgname{Johannes Kepler University Linz}, \orgaddress{\country{Austria}}}

\corres{Corresponding author Andreas M\"uller, \email{a.mueller@jku.at}}

\presentaddress{Altenberger Str. 69, 4040 Linz}

\abstract[Abstract]{
The Lie group $SE\left( 3\right) $ of isometric orientation preserving
transformation is used for modeling multibody systems, robots, and Cosserat
continua. The use of these models in numerical simulation and optimization
schemes necessitates the exponential map, its right-trivialized differential
(often referred to as tangent operator), as well as higher derivatives in
closed form. The $6\times 6$ matrix representation of the differential, $%
\mathbf{dexp}_{\mathbf{X}}:se\left( 3\right) \rightarrow se\left( 3\right) $%
, and its first derivative were reported using a $3\times 3$ block
partitioning. In this paper, the differential, its first and second
derivative, as well as the Jacobian and Hessian of the evaluation maps, $%
\mathbf{dexp}_{\mathbf{X}}\mathbf{Z}$ and $\mathbf{dexp}_{\mathbf{X}}^{T}%
\mathbf{Z}$, are reported avoiding the block partitioning. For all of them,
higher-order approximations are derived. Besides the compactness, the
advantage of the presented closed form relations is their numerical
robustness when combined with the local approximation. The formulations are
demonstrated for computation of the deformation field and the strain rates
of an elastic Cosserat-Simo-Reissner rod.
}



\maketitle

\renewcommand\thefootnote{\fnsymbol{footnote}}
\setcounter{footnote}{1}

\parindent0pt \parskip2pt \setcounter{topnumber}{9} %
\setcounter{bottomnumber}{9} \renewcommand{\textfraction}{0.00001}

\renewcommand {\floatpagefraction}{0.999} \renewcommand{\textfraction}{0.01} %
\renewcommand{\topfraction}{0.999} \renewcommand{\bottomfraction}{0.99} %
\renewcommand{\floatpagefraction}{0.99} \setcounter{totalnumber}{9}

\section{Introduction}

Computational multibody system (MBS) dynamics \cite%
{EngoMarthinsen2001,BrulsCardonaArnold2012,SonnevilleCardonaBruls2014,TerzeMueller2016,ChevallierLerbet2017}
and modern robotics \cite{LynchPark2017,MurrayBook,SeligBook} makes
extensive use of the special Euclidean group, denoted with $SE\left(
3\right) $. This Lie group provides the framework for describing spatial
frame motions, which can be used to represent rigid body motions but also
the deformation field of continua within the micropolar theory \cite%
{Cosserat1909,TruesdellToupin1960}, such as Cosserat beam and shell models 
\cite%
{Simo1985,BorriBottasso1994a,BorriBottasso1994b,SonnevilleCardonaBruls2014,SonnevillePhD,BauchauSonneville_Shells2021,HerrmannKotyczka2024,Ge2025}%
. The latter has become a key concept in soft robotics \cite%
{ArmaniniBoyer2023}. The exponential map $\exp :se\left( 3\right)
\rightarrow $ $SE\left( 3\right) $ provides a natural parameterization in
terms of canonical coordinates, which are identical to the instantaneous
screw coordinates associated to the frame motion, a concept known from
theoretical kinematics \cite{BottemaRoth1979,McCarthyBook1990}. Robot
control as well as computational methods for MBS dynamics necessitate the
derivative of the coordinate map. On $SE\left( 3\right) $, being a Lie
group, this boils down to the differential of the exponential map evaluated
at the origin. A typical application of this trivialized differential is to
express the rigid body twist in terms of time derivatives of the
coordinates, or to compute the deformation field in a geometrically exact
rod model. Moreover, the first and second derivatives (respectively Jacobian
and Hessian) of terms involving the differential are needed within optimal
control and numerical integration schemes, e.g. implicit generalized $\alpha 
$ methods \cite%
{KryslEndres2005,ArnoldBrulsCardona2015,ArnoldHante2017,BrulsCardona2010,BrulsCardonaArnold2012}%
. Higher derivatives are required in other applications, such as
higher-order smooth interpolation of spatial motions or deformations \cite%
{ParkRavani1997,POESplines2025}, or simply for computing rigid body
acceleration and jerk, which is necessary for trajectory planning with
controllable smoothness \cite{Reiter-TIT2018}.

Closed form expressions for the first derivative of the differential of the
exponential map where reported in the literature \cite%
{ParkChung2005,SonnevillePhD,RSPA2021}. These formulations exploited the
semidirect product topology of $SE\left( 3\right) $, which leads to a $%
3\times 3$ block partitioning with separate expressions for translations and
rotations. This leads to complicated and computationally complex
formulations, however. Within numerical Lie group integration schemes and
for approximation of curves on Lie groups, the use of approximations of the
exponential map was proposed. To this end, truncated series expansions of
the exponential on $SE\left( 3\right) $ and of its derivative were employed.
No systematic approach has been introduced, however. The approximation of
the exponential map on a general matrix Lie group was addressed in \cite%
{ZannaMuntheKaas2002}\cite{IserlesZanna2005}, and on a semi-direct product
Lie group in \cite{NobariHosseini2010}.

The contribution of this paper is three-fold. Firstly, compact \emph{%
non-partitioned} closed form expressions for higher derivatives of the
exponential map on $SE\left( 3\right) $ in terms of $6\times 6$ matrices are
derived in Sec. \ref{secDerivatives}, avoiding explicit partitioning in $%
3\times 3$ matrices. This includes derivatives of the exponential map as
well the Jacobian and Hessian of its evaluation map, which are crucial for
optimization and dynamics simulation of flexible MBS. The limit values for
all expressions at the parameterization singularity are also reported. 
\color[rgb]{0,0,0}%
The main advantage of the $6\times 6$ representation is that it yields
compact expressions that are simpler to implement. 
\color{black}%
Secondly, higher-order approximations of all above relations are derived in
Sec. \ref{secApprox1} and \ref{secApprox2}. They are crucial for robust
numerical evaluation of the closed form relations, but are also advantageous
within Lie group integration and approximation schemes. Thirdly, an
exhaustive summary of already reported closed form expressions for higher
derivatives using the block partitioned representation, including their
limit values, is reported in appendix \ref{secAppSE3Block}. 
\color[rgb]{0,0,0}%
A Mathematica$^{\copyright }$ library comprising all presented formulae is
available at
https://github.com/andreasmuellerjku/SO3-SE3-exp-Derivatives-Jacobian-Hessian%
\color{black}%

\section{Special Euclidean Group and Frame Transformations in $E^{3}$\label%
{secSE3Block}}

\subsection{Preliminaries}

Let $G$ be an $n$-dimensional Lie group with Lie algebra $\mathfrak{g}$. Lie
algebra elements are denoted $\hat{\mathbf{X}}\in \mathfrak{g}$, and when
represented as vectors, they are denoted $\mathbf{X}\in {\mathbb{R}}^{n}$,
which implies an obvious isomorphism. The exponential map $\exp :\mathfrak{g}%
\rightarrow G$ induces a parameterization of $G$ in terms of canonical
coordinates (of first kind). Its right-trivialized differential $\mathrm{dexp%
}_{\mathbf{X}}:\mathfrak{g}\rightarrow \mathfrak{g}$ is defined as%
\begin{equation}
\mathrm{dexp}_{\hat{\mathbf{X}}}(\hat{\mathbf{Y}})=\left( \mathrm{D}_{\hat{%
\mathbf{X}}}\exp \right) 
\hspace{-0.5ex}%
(\hat{\mathbf{Y}})\exp (-\hat{\mathbf{X}})  \label{diff1}
\end{equation}%
where $\left( \mathrm{D}_{\hat{\mathbf{X}}}\exp \right) 
\hspace{-0.5ex}%
(\hat{\mathbf{Y}}):=\frac{d}{dt}\exp (\hat{\mathbf{X}}+t\hat{\mathbf{Y}}%
)|_{t=0}$ is the directional derivative $\mathrm{D}_{\hat{\mathbf{X}}}\exp :%
\mathfrak{g}\rightarrow T_{\exp \hat{\mathbf{X}}}G$ of $\exp $ at $\hat{%
\mathbf{X}}$ in direction of $\hat{\mathbf{Y}}$. Let $g\left( t\right) =\exp 
\hat{\mathbf{X}}\left( t\right) $ be a curve in $G$. Then $\dot{g}g^{-1}=%
\mathrm{dexp}_{\hat{\mathbf{X}}}(\hat{\dot{\mathbf{X}}})\in \mathfrak{g}$ is
the corresponding right-invariant vector field on $\mathfrak{g}$. The left
invariant vector field is obtained as $g^{-1}\dot{g}=\mathrm{dexp}_{-\hat{%
\mathbf{X}}}(\hat{\dot{\mathbf{X}}})\in \mathfrak{g}$ (notice the negative
sign of the argument), where $\mathrm{dexp}_{-\hat{\mathbf{X}}}$ is
occasionally called the left-trivialized differential. The right-trivialized
dexp map is known as Magnus expansion \cite{Iserles1984}, which was
originally reported by Hausdorff \cite[pp. 26 \& 36ff]{Hausdorff1906}, and
later by Magnus \cite{Magnus1954} and Varadarajan \cite[Theorem 2.14.3.]%
{Varadarajan1984}.

\subsection{Block Partitioned Representation of Semi-Direct Product Group}

The special Euclidean group is the semi-direct product group $SE\left(
3\right) =SO\left( 3\right) \ltimes {\mathbb{R}}^{3}$. This is encoded in
the matrix representation of $SE\left( 3\right) $ in terms of $4\times 4$
matrices (often referred to 'homogenous' transformation matrices)%
\begin{equation}
\mathbf{H}=%
\begin{bmatrix}
\mathbf{R} & \mathbf{r} \\ 
\mathbf{0} & 1%
\end{bmatrix}%
\in SE\left( 3\right)
\end{equation}%
which describes the transformation of point coordinates in $E^{3}$. Matrix $%
\mathbf{H}$ also serves to represent the relative configuration of two
frames, and thus of a Euclidean motion, described by the rotation matrix $%
\mathbf{R}\in SO\left( 3\right) $ and displacement vector $\mathbf{r}\in {%
\mathbb{R}}^{3}$. The Lie algebra $se\left( 3\right) $ consists of matrices
of the form%
\begin{equation}
\hat{\mathbf{X}}=%
\begin{bmatrix}
\tilde{\mathbf{x}} & \mathbf{y} \\ 
\mathbf{0} & 0%
\end{bmatrix}%
\in se\left( 3\right)  \label{Xhat}
\end{equation}%
where the notation $\tilde{\mathbf{x}}\in so\left( 3\right) $ is used to
assign the skew symmetric matrix to the vector $\mathbf{x}\in \mathbb{R}^{3}$%
. There is an obvious isomorphism $se\left( 3\right) =so\left( 3\right)
\ltimes {\mathbb{R}}^{3}\cong {\mathbb{R}}^{3}\times \mathbb{R}^{3}$, where
matrix $\hat{\mathbf{X}}\in se\left( 3\right) $ is given in terms of vector $%
\mathbf{X}=\left( \mathbf{x},\mathbf{y}\right) \in {\mathbb{R}}^{6}$.
Capital letters are used to denote canonical coordinates on $se\left(
3\right) $, while small letters are used to denote the canonical coordinates
on the respective factor, i.e. $\mathbf{x}\in {\mathbb{R}}^{3}\cong so\left(
3\right) $ and $\mathbf{y}\in {\mathbb{R}}^{3}$ are coordinates on $SO\left(
3\right) $ and ${\mathbb{R}}^{3}$, respectively. The exponential map on $%
SE\left( 3\right) $ can be expressed explicitly as 
\begin{equation}
\exp (\hat{\mathbf{X}})=%
\begin{bmatrix}
\exp \tilde{\mathbf{x}} & \mathbf{dexp}_{\mathbf{x}}\mathbf{y} \\ 
\mathbf{0} & 1%
\end{bmatrix}
\label{expSE3}
\end{equation}%
where $\exp \tilde{\mathbf{x}}$ is the exponential map on $SO\left( 3\right) 
$, and $\mathbf{dexp}_{\mathbf{x}}\mathbf{y}$ is the matrix form of its
right-trivialized differential. Kinematically, $\mathbf{H}=\exp (\hat{%
\mathbf{X}})$ describes a screw motion in $E^{3}$, and $\mathbf{X}=\left( 
\mathbf{x},\mathbf{y}\right) $ is called the instantaneous screw coordinate
vector, where $\mathbf{x}$ is the scaled rotation vector (unit vector along
the rotation axis times the angle), also called Rodrigues vector.

Closed form expressions for the exp map and its differential on $SO\left(
3\right) $ are \cite{RSPA2021}%
\begin{eqnarray}
\exp \tilde{\mathbf{x}} &=&\mathbf{I}+\alpha \tilde{\mathbf{x}}+\tfrac{1}{2}%
\beta \tilde{\mathbf{x}}^{2}  \label{expSO3} \\
\mathbf{dexp}_{\mathbf{x}} &=&\mathbf{I}+\tfrac{\beta }{2}\tilde{\mathbf{x}}%
+\delta \tilde{\mathbf{x}}^{2}  \label{dexpSO31} \\
&=&\mathbf{I}+\tfrac{\beta }{2}\tilde{\mathbf{x}}+\left( 1-\alpha \right) 
\tilde{\mathbf{n}}^{2}  \label{dexpSO32} \\
\mathbf{dexp}_{\mathbf{x}}^{-1} &=&\mathbf{I}-\tfrac{1}{2}\tilde{\mathbf{x}}+%
\tfrac{1}{\varphi ^{2}}\left( 1-\gamma \right) \tilde{\mathbf{x}}^{2}
\label{dexpInvSO31} \\
&=&\mathbf{I}-\tfrac{1}{2}\tilde{\mathbf{x}}+\left( 1-\gamma \right) \tilde{%
\mathbf{n}}^{2}  \label{dexpInvSO32}
\end{eqnarray}%
with $\mathbf{n}:=\mathbf{x/}\left\Vert \mathbf{x}\right\Vert $ being the
unit vector along the instantaneous rotation axis, and%
\begin{equation}
\alpha :=\mathrm{sinc}\varphi ,\ \ \beta :=\mathrm{sinc}^{2}\mathrm{\,}%
\tfrac{\varphi }{2},\ \gamma :=\tfrac{\alpha }{\beta },\ \ \delta :=\tfrac{%
1-\alpha }{\varphi ^{2}}  \label{alpha}
\end{equation}%
where the Euclidean norm $\varphi :=\left\Vert \mathbf{x}\right\Vert $ of $%
\mathbf{x}\in {\mathbb{R}}^{3}$ denotes the angle of rotation. Expression (%
\ref{expSO3}) is referred to as the Euler-Rodrigues formula. Using the
vector representation of $se\left( 3\right) \cong {\mathbb{R}}^{3}\times 
\mathbb{R}^{3}$, the differential of the exp map on $SE(3)$, and its
inverse, admit the $3\times 3$ block partitioned matrix form\footnote{%
\color[rgb]{0,0,0}Throughout the paper, the simplified notations $\exp 
\tilde{\mathbf{x}}=\exp \mathbf{x}$, $\exp (\hat{\mathbf{X}})=\exp (\mathbf{X%
})$, $\mathbf{dexp}_{\tilde{\mathbf{x}}}=\mathbf{dexp}_{\mathbf{x}}$, $%
\mathbf{dexp}_{\hat{\mathbf{X}}}=\mathbf{dexp}_{\mathbf{X}}$ etc are used.}%
\begin{equation}
\mathbf{dexp}_{\mathbf{X}}=%
\begin{bmatrix}
\mathbf{dexp}_{\mathbf{x}} & \mathbf{0} \\ 
\left( \mathrm{D}_{\mathbf{x}}\mathbf{dexp}\right) 
\hspace{-0.5ex}%
\left( \mathbf{y}\right) & \mathbf{dexp}_{\mathbf{x}}%
\end{bmatrix}%
,\ \ \ \mathbf{dexp}_{\mathbf{X}}^{-1}=%
\begin{bmatrix}
\mathbf{dexp}_{\mathbf{x}}^{-1} & \mathbf{0} \\ 
\left( \mathrm{D}_{\mathbf{x}}\mathbf{dexp}^{-1}\right) 
\hspace{-0.5ex}%
\left( \mathbf{y}\right) & \mathbf{dexp}_{\mathbf{x}}^{-1}%
\end{bmatrix}%
.  \label{dexpSE3Block}
\end{equation}%
The apparent block partitioning is inherited from (\ref{expSE3}). Closed
form expressions of the terms in (\ref{dexpSE3Block}) where reported \cite%
{ParkChung2005,SonnevillePhD,RSPA2021}. A comprehensive summary is given in
appendix \ref{secAppSE3Block}, including the limit values for $\mathbf{x}%
\rightarrow \mathbf{0}$.

The significance of the differential is that it relates frame velocity to
the change of canonical coordinates. The spatial angular velocity, defined
as the vector field $\tilde{\boldsymbol{\omega }}=\dot{\mathbf{R}}\mathbf{R}%
^{T}\in so\left( 3\right) $, is expressed with (\ref{dexpSO31}) as $%
\boldsymbol{\omega }=\mathbf{dexp}_{\mathbf{x}}\dot{\mathbf{x}}$. The
spatial velocity screw (twist) for a general rigid body motion is defined as 
$\hat{\mathbf{V}}=\dot{\mathbf{H}}\mathbf{H}^{-1}=\left[ 
\begin{smallmatrix}
\tilde{\boldsymbol{\omega }} & \mathbf{v} \\ 
\mathbf{0} & 0%
\end{smallmatrix}%
\right] \in se\left( 3\right) $, which corresponds to the vector $\mathbf{V}=%
\left[ 
\begin{smallmatrix}
\boldsymbol{\omega } \\ 
\mathbf{v}%
\end{smallmatrix}%
\right] \in {\mathbb{R}}^{3}\times \mathbb{R}^{3}$. It is related to the
change of (canonical) screw coordinates by $\mathbf{V}=\mathbf{dexp}_{%
\mathbf{X}}\dot{\mathbf{X}}$. The latter is called the \emph{local
reconstruction equation} since its solution yields the (local) coordinates $%
\mathbf{X}\left( t\right) $, and thus the motion, reconstructed from the
twist $\mathbf{V}\left( t\right) $. In this context, the equation $\dot{%
\mathbf{H}}=\hat{\mathbf{V}}\mathbf{H}$ is referred to as Poisson--Darboux
equation \cite{ConduracheAAS2018}. The body-fixed representation of twist is
defined as $\hat{\mathbf{V}}=\mathbf{H}^{-1}\dot{\mathbf{H}}$, and related
to the coordinate rate as $\mathbf{V}=\mathbf{dexp}_{-\mathbf{X}}\dot{%
\mathbf{X}}$.

\subsection{Application Examples\label{secApplications}}

\color[rgb]{0,0,0}%
A typical application of the \textbf{dexp} map in kinematics of rigid and
flexible bodies is to express the twist and deformation field, respectively.
The higher-kinematics involves the derivatives, Jacobian, and Hessian. This
is briefly outlines in the following. Numerical examples are presented in
Sec. \ref{secApprox} and \ref{secCosserat}.%
\color{black}%

\subsubsection{Acceleration and Jerk of a Rigid Body}

The kinematic problem of describing rigid body motions further includes
derivatives of the twist (called reduced acceleration). In terms of the time
derivative of the canonical coordinates, the first and second time
derivative of the spatial twist are%
\begin{eqnarray}
\dot{\mathbf{V}} &=&\mathbf{dexp}_{\mathbf{X}}\ddot{\mathbf{X}}+\left( 
\mathrm{D}_{\hat{\mathbf{X}}}\mathbf{dexp}\right) (\dot{\mathbf{X}})\dot{%
\mathbf{X}}  \label{Vd} \\
\ \ \ \ \ \ \ \ \ \ddot{\mathbf{V}} &=&\mathbf{dexp}_{\mathbf{X}}\dddot{%
\mathbf{X}}+2(\mathrm{D}_{\mathbf{X}}\mathbf{dexp})(\dot{\mathbf{X}})\ddot{%
\mathbf{X}}+(\mathrm{D}_{\mathbf{X}}\mathbf{dexp})(\ddot{\mathbf{X}})\dot{%
\mathbf{X}}+(\mathrm{D}_{\mathbf{X}}^{2}\mathbf{dexp})(\dot{\mathbf{X}})(%
\dot{\mathbf{X}})\dot{\mathbf{X}}.  \label{Vdd}
\end{eqnarray}%
In addition to the directional derivative $\left( \mathrm{D}_{\hat{\mathbf{X}%
}}\mathbf{dexp}\right) (\mathbf{U})$ of $\mathbf{dexp}_{\hat{\mathbf{X}}}$
along $\mathbf{U}\in \mathbb{R}^{6}$ the last term includes the repeated
directional derivative $(\mathrm{D}_{\mathbf{X}}^{2}\mathbf{dexp})(\mathbf{U}%
)(\mathbf{S})$ along $\mathbf{S}\in \mathbb{R}^{6}$.

\subsubsection{Gradient and Hessian of Elastic Potential of Non-Linear Rod
Models}

A geometrically exact rod is a 1-dimensional continuum where the cross
section exhibits spatial motions. Under the Bernoulli assumption of constant
cross section, the rod is kinematically represented by the motion of a cross
section frame as a function of the arc length. The rod kinematics thus
resembles a rigid body motion where the pose $\mathbf{H}\left( \tau \right)
\in SE(3)$ of a cross section frame describes the displacement field, and
depends on the (normalized) arc length $\tau \in \left[ 0,1\right] $. This
is referred to as Simo-Reissner \cite{Reissner1972,Reissner1973,Simo1985}
and Cosserat rod \cite{Cosserat1970,Antman2005} theory. A deformation
measure is defined as $\hat{\boldsymbol{\chi }}=\mathbf{H}^{-1}\mathbf{H}%
^{\prime }=\left[ 
\begin{smallmatrix}
\tilde{\boldsymbol{\kappa }} & \boldsymbol{\rho } \\ 
\mathbf{0} & 0%
\end{smallmatrix}%
\right] \in se\left( 3\right) $ \cite{ArmaniniBoyer2023} resembling the
rigid body twist above. This is a (left-invariant) body-fixed definition,
i.e. invariant under the change of global reference. The rod displacement is
parameterized in terms of screw coordinates $\mathbf{H}\left( \tau \right)
=\exp \hat{\mathbf{X}}(\tau )$, now depending on the arc length $\tau $. The
deformation is thus%
\begin{equation}
\boldsymbol{\chi }\left( \tau \right) =\mathbf{dexp}_{-\mathbf{X}\left( \tau
\right) }\mathbf{X}^{\prime }\left( \tau \right)  \label{RecSE3}
\end{equation}%
where the negative sign is due to the body-fixed representation. In context
of flexible MBS dynamics, the differential is called tangent operator \cite%
{BrulsCardonaArnold2012}, which usually refers to the right-trivialized
differential with the negative sign in (\ref{RecSE3}).

Central objects in computational flexible MBS dynamics are the directional
derivative of $\boldsymbol{\chi }$ and the deformation gradient $\frac{%
\partial \boldsymbol{\chi }}{\partial \mathbf{X}}=\frac{\partial }{\partial 
\mathbf{X}}\mathbf{\frac{\partial }{\partial \mathbf{X}}}\left( \mathbf{dexp}%
_{-\mathbf{X}}\mathbf{X}^{\prime }\right) $ \cite%
{BauchauBook2010,SonnevilleCardonaBruls2014}. Further relevant is the strain
rate, i.e. derivatives (\ref{Vd}) using (\ref{RecSE3}) with $\dot{\mathbf{V}}
$ replaced by $\mathbf{V}^{\prime }$.

Assuming linear Hookean elastic constitutive relations, the elastic
potential density $\bar{V}=\frac{1}{2}\left( \boldsymbol{\chi }-\boldsymbol{%
\chi }_{0}\right) ^{T}\mathbf{K}\left( \boldsymbol{\chi }-\boldsymbol{\chi }%
_{0}\right) $ is introduced, where $\mathbf{K}\left( \tau \right) $ is the
stiffness matrix related to the cross section at $\tau $, and $\boldsymbol{%
\chi }_{0}$ describes the undeformed geometry. In various contexts, the
gradient and the Hessian of the potential are needed. When $\mathbf{H}\left(
\tau \right) $ is parameterized with $\mathbf{X}\left( \tau \right) $, they
can be written as%
\begin{eqnarray}
\ \ \ \frac{\partial \bar{V}}{\partial \mathbf{X}} &=&\left( \boldsymbol{%
\chi }-\boldsymbol{\chi }_{0}\right) \mathbf{K}\frac{\partial \boldsymbol{%
\chi }}{\partial \mathbf{X}}=\left( \boldsymbol{\chi }-\boldsymbol{\chi }%
_{0}\right) \mathbf{K\frac{\partial }{\partial \mathbf{X}}}\left( \mathbf{%
dexp}_{-\hat{\mathbf{X}}}\mathbf{X}^{\prime }\right)  \label{VGrad} \\
\ \ \ \ \ \ \ \ \frac{\partial ^{2}\bar{V}}{\partial \mathbf{X}^{2}} &=&%
\frac{\partial ^{2}}{\partial \mathbf{X}^{2}}\left( \mathbf{a}^{T}\mathbf{%
dexp}_{-\mathbf{X}\left( \tau \right) }\mathbf{X}^{\prime }\left( \tau
\right) \right) +\left( \frac{\partial }{\partial \mathbf{X}}\mathbf{dexp}_{-%
\mathbf{X}}\mathbf{X}^{\prime }\right) ^{T}\mathbf{K}\frac{\partial }{%
\partial \mathbf{X}}\mathbf{dexp}_{-\mathbf{X}}\mathbf{X}^{\prime }
\label{VHesse}
\end{eqnarray}%
where $\mathbf{a}:=\mathbf{K}\left( \boldsymbol{\chi }-\boldsymbol{\chi }%
_{0}\right) $ is inserted after differentiation. Computing (\ref{VGrad}) and
(\ref{VHesse}) involves the first partial derivatives of the product $%
\mathbf{dexp}_{-\mathbf{X}}\mathbf{X}^{\prime }$ and the Hessian of $\mathbf{%
a}^{T}\mathbf{dexp}_{-\mathbf{X}}\mathbf{X}^{\prime }$ with given $\mathbf{a}
$ and $\mathbf{X}^{\prime }$.

It becomes obvious from the relations summarized in appendix \ref%
{secAppSE3Block} that the explicit expressions become involved, and further
that their robust numerical evaluation is challenging. To address this
issue, closed form compact relations are derived in the following. Key to
these formulations is to avoid explicit use of the block partitioning.

\section{Non-Partitioned $6\times 6$ Matrix Representation}

A matrix in (\ref{Xhat}) satisfies its characteristic equation $\hat{\mathbf{%
X}}^{4}+\left\Vert \mathbf{x}\right\Vert ^{2}\hat{\mathbf{X}}=\mathbf{0}$,
with $\left\Vert \mathbf{x}\right\Vert $ being the norm of $\mathbf{x}\in 
\mathbb{R}^{3}$. Inserted in the series expansion $\exp (\hat{\mathbf{X}}%
)=\sum_{i=0}^{\infty }\tfrac{1}{i!}\hat{\mathbf{X}}$ yields \cite{SeligBook}%
\begin{equation}
\exp (\hat{\mathbf{X}})=\mathbf{I}+\hat{\mathbf{X}}+\tfrac{1}{2}\beta \hat{%
\mathbf{X}}^{2}+\delta \hat{\mathbf{X}}^{3}
\end{equation}%
that replaces (\ref{expSE3}). The matrix in (\ref{dexpSE3Block}), and its
inverse, admit the series expansions 
\begin{equation}
\mathbf{dexp}_{\hat{\mathbf{X}}}=\sum_{i=0}^{\infty }\frac{1}{\left(
i+1\right) !}\mathbf{ad}_{\hat{\mathbf{X}}}^{i},\ \ \mathbf{dexp}_{\hat{%
\mathbf{X}}}^{-1}=\sum_{i=0}^{\infty }\frac{B_{i}}{i!}\mathbf{ad}_{\hat{%
\mathbf{X}}}^{i}  \label{dexpMat}
\end{equation}%
in terms of the $6\times 6$ adjoined operator matrix 
\begin{equation}
\mathbf{ad}_{\hat{\mathbf{X}}}=%
\begin{bmatrix}
\tilde{\mathbf{x}} & \mathbf{0} \\ 
\tilde{\mathbf{y}} & \tilde{\mathbf{x}}%
\end{bmatrix}
\label{adX}
\end{equation}%
\color[rgb]{0,0,0}%
where $B_{i}$ are the Bernoulli numbers. In particular, $B_{0}=1,B_{1}=-%
\tfrac{1}{2},B_{2}=\tfrac{1}{6},B_{4}=-\tfrac{1}{30},B_{6}=\tfrac{1}{46}$,
and $B_{2n+1}=0,n\geq 1$.%
\color{black}
The expressions (\ref{dexpMat}) are classical results presented for general
quadratic Lie groups \cite[pp. 26 \& 36ff]{Hausdorff1906},\cite[Theorem
2.14.3.]{Varadarajan1984},\cite{Iserles1984}. On $SE(3)$, the adjoined
operator matrix (\ref{adX}) satisfies the identity $\mathbf{ad}_{\hat{%
\mathbf{X}}}^{5}=-2\left\Vert \mathbf{x}\right\Vert ^{2}\mathbf{ad}_{\hat{%
\mathbf{X}}}^{3}$ $-\left\Vert \mathbf{x}\right\Vert ^{4}\mathbf{ad}_{\hat{%
\mathbf{X}}}$, which is deduced from its characteristic polynomial $\det
\left( \mathbf{I}-\lambda \mathbf{ad}_{\hat{\mathbf{X}}}\right) =\lambda
^{6}+2\left\Vert \mathbf{x}\right\Vert ^{2}\lambda ^{4}+\left\Vert \mathbf{x}%
\right\Vert ^{4}\lambda ^{2}$. Thus relations (\ref{dexpMat}) involve powers
of $\mathbf{ad}_{\hat{\mathbf{X}}}$ up to $k=4$ only. For sake of brevity,
the $i$-th power of the adjoint operator matrix is denoted%
\begin{equation}
\mathbf{P}_{i}\left( \mathbf{X}\right) :=\mathbf{ad}_{\mathbf{X}}^{i}
\label{Pi}
\end{equation}%
which satisfies $\mathbf{P}_{i}\left( \mathbf{X}\right) =\mathbf{ad}_{%
\mathbf{X}}\mathbf{P}_{i-1}\left( \mathbf{X}\right) =\mathbf{P}_{i-1}\left( 
\mathbf{X}\right) \mathbf{ad}_{\mathbf{X}}$. Therewith, explicit forms of
the right-trivialized differential, and its inverse, on $SE\left( 3\right) $
are obtained as%
\begin{eqnarray}
\mathbf{dexp}_{\hat{\mathbf{X}}} &=&\mathbf{I}+a_{1}\mathbf{P}_{1}\left( 
\mathbf{X}\right) +a_{2}\mathbf{P}_{2}\left( \mathbf{X}\right) +a_{3}\mathbf{%
P}_{3}\left( \mathbf{X}\right) +a_{4}\mathbf{P}_{4}\left( \mathbf{X}\right)
\label{dexpSE3} \\
\mathbf{dexp}_{\hat{\mathbf{X}}}^{-1} &=&\mathbf{I}+b_{1}\mathbf{P}%
_{1}\left( \mathbf{X}\right) +b_{2}\mathbf{P}_{2}\left( \mathbf{X}\right)
+b_{4}\mathbf{P}_{4}\left( \mathbf{X}\right)  \label{dexpInvSE3}
\end{eqnarray}%
where the coefficients $a_{i},b_{i}$ are in general analytic functions of $%
\mathbf{X}$. The non-zero coefficients are derived as%
\begin{eqnarray}
a_{1} &=&\beta -\tfrac{\alpha }{2},\ \ a_{2}=\tfrac{1}{2}\left( 5\delta -%
\tfrac{\beta }{2}\right) =\tfrac{1}{\varphi ^{2}}\left( \tfrac{5}{2}-\tfrac{5%
}{2}\alpha \right) -\tfrac{\beta }{4}  \notag \\
a_{3} &=&\tfrac{1}{2\varphi ^{2}}\left( \beta -\alpha \right) ,\ \ a_{4}=%
\tfrac{1}{2\varphi ^{2}}\left( 3\delta -\tfrac{\beta }{2}\right) =\tfrac{1}{%
2\varphi ^{4}}\left( 3-3\alpha \right) -\tfrac{\beta }{4\varphi ^{2}}
\label{ab} \\
b_{1} &=&-\tfrac{1}{2},\ \ b_{2}=\tfrac{1}{\varphi ^{2}}\left( 2-\tfrac{%
1+3\alpha }{2\beta }\right) ,\ b_{4}=\tfrac{1}{\varphi ^{4}}\left( 1-\tfrac{%
1+\alpha }{2\beta }\right) .  \notag
\end{eqnarray}%
The dependence of $a_{i}$ and $b_{i}$ on $\varphi $ only is due to the
semi-direct product structure of $SE\left( 3\right) $.

\begin{remark}
The closed form expressions can be robustly evaluated when $\left\Vert 
\mathbf{x}\right\Vert \rightarrow 0$ since the coefficients $\alpha ,\beta
,\gamma $ in (\ref{alpha}) are defined in terms of $\mathrm{sinc}$. The
expressions (\ref{dexpSO31}) and (\ref{dexpInvSO31}) for the differential
still involve division by $\left\Vert \mathbf{x}\right\Vert $, which may
lead to numerical issues (although $\delta $ and $\left( 1-\gamma \right)
/\varphi ^{2}$ converge to 0 for $\left\Vert \mathbf{x}\right\Vert
\rightarrow 0$). This is alleviated by using the normalized vector $\mathbf{n%
}$ in (\ref{dexpSO32}) and (\ref{dexpInvSO32}), which can be computed
robustly. Also some of the coefficients (\ref{ab}) involve division by $%
\left\Vert \mathbf{x}\right\Vert $. A computationally advantageous form
avoiding these divisions is obtained in terms of $\mathbf{N}:=\mathbf{X/}%
\left\Vert \mathbf{x}\right\Vert $. This yields the alternative expressions
of (\ref{dexpSE3}) and (\ref{dexpInvSE3})%
\begin{eqnarray*}
\mathbf{dexp}_{\hat{\mathbf{X}}} &=&\mathbf{I}+\left( \beta -\tfrac{\alpha }{%
2}\right) \mathbf{ad}_{\hat{\mathbf{X}}}-\tfrac{\beta }{4}\mathbf{ad}_{\hat{%
\mathbf{X}}}^{2}+\left( \tfrac{5}{2}-\tfrac{5}{2}\alpha \right) \mathbf{ad}_{%
\hat{\mathbf{N}}}^{2} \\
&&+\tfrac{1}{2}\left( \beta -\alpha \right) \mathbf{ad}_{\hat{\mathbf{N}}}%
\mathbf{ad}_{\hat{\mathbf{X}}}-\tfrac{\beta }{4}\mathbf{ad}_{\hat{\mathbf{N}}%
}^{2}\mathbf{ad}_{\hat{\mathbf{X}}}^{2}+\left( \tfrac{3}{2}-\tfrac{3}{2}%
\alpha \right) \mathbf{ad}_{\hat{\mathbf{N}}}^{4} \\
\mathbf{dexp}_{\hat{\mathbf{X}}}^{-1} &=&\mathbf{I}-\tfrac{1}{2}\mathbf{ad}_{%
\hat{\mathbf{X}}}+\left( 2-\tfrac{1+3\alpha }{2\beta }\right) \mathbf{ad}_{%
\hat{\mathbf{N}}}^{2}+\left( 1-\tfrac{1+\alpha }{2\beta }\right) \mathbf{ad}%
_{\hat{\mathbf{N}}}^{4}.
\end{eqnarray*}
\end{remark}

\section{Derivatives of the Differential in Non-Partitioned Representation}

\label{secDerivatives}

\subsection{Preliminaries}

The directional derivative of the matrix representation (\ref{dexpMat}) of
the right-trivialized differential along $\mathbf{U}=\left( \mathbf{u},%
\mathbf{v}\right) \in {\mathbb{R}}^{6}$ is denoted $\left( \mathrm{D}_{%
\mathbf{X}}\mathbf{dexp}\right) 
\hspace{-0.5ex}%
\left( \mathbf{U}\right) =\frac{d}{dt}\mathbf{dexp}_{\mathbf{X}+t\mathbf{U}%
}|_{t=0}$, and that of its inverse as $\left( \mathrm{D}_{\mathbf{X}}\mathbf{%
dexp}^{-1}\right) 
\hspace{-0.5ex}%
\left( \mathbf{U}\right) =\frac{d}{dt}\mathbf{dexp}_{\mathbf{X}+t\mathbf{U}%
}^{-1}|_{t=0}$. The second directional derivative along the vector $\mathbf{S%
}=\left( \mathbf{s},\mathbf{r}\right) \in {\mathbb{R}}^{3}\times {\mathbb{R}}%
^{3}$ is denoted with $\left( \mathrm{D}_{\mathbf{X}}^{2}\mathbf{dexp}%
\right) 
\hspace{-0.5ex}%
\left( \mathbf{U}\right) \left( \mathbf{S}\right) =\frac{d}{dt}\left( 
\mathrm{D}_{\mathbf{X}+t\mathbf{S}}\mathbf{dexp}\right) 
\hspace{-0.5ex}%
\left( \mathbf{U}\right) |_{t=0}$, and analogously for its inverse. The
following result is crucial for deriving closed form and recursive
expressions of the directional derivatives.

\begin{lemma}
\label{lemDerPi}The directional derivative of $\mathbf{P}_{i}$ along $%
\mathbf{U}=\left( \mathbf{u},\mathbf{v}\right) $ admits the following closed
form expression as well as the series expansion%
\vspace{-2ex}%
\begin{equation}
\left( \mathrm{D}_{\mathbf{X}}\mathbf{P}_{i}\right) 
\hspace{-0.5ex}%
\left( \mathbf{U}\right) =\sum_{j=0}^{i-1}\mathbf{P}_{j}\left( \mathbf{X}%
\right) \mathbf{ad}_{\mathbf{U}}\mathbf{P}_{i-j-1}\left( \mathbf{X}\right)
=\left( \mathrm{D}_{\mathbf{X}}\mathbf{P}_{i-1}\right) 
\hspace{-0.5ex}%
\left( \mathbf{U}\right) \mathbf{ad}_{\mathbf{X}}+\mathbf{P}_{i-1}\left( 
\mathbf{X}\right) \mathbf{ad}_{\mathbf{U}}.  \label{DPi}
\end{equation}
\end{lemma}

\begin{proof}
The statement follows from the definition of (\ref{Pi}), noting that $(%
\mathrm{D}_{\mathbf{X}}\mathbf{ad}_{\mathbf{X}})%
\hspace{-0.5ex}%
\left( \mathbf{U}\right) =\mathbf{ad}_{\mathbf{U}}$, and the
non-commutativity of products of the form $\mathbf{ad}_{\mathbf{X}}^{k}%
\mathbf{ad}_{\mathbf{U}}^{l}$:%
\begin{align}
\left( \mathrm{D}_{\mathbf{X}}\mathbf{P}_{i}\right) 
\hspace{-0.5ex}%
\left( \mathbf{U}\right) =& (\mathrm{D}_{\mathbf{X}}\mathbf{ad}_{\mathbf{X}%
}^{i})%
\hspace{-0.5ex}%
\left( \mathbf{U}\right) =\frac{d}{dt}\left. \left( \mathbf{ad}_{\mathbf{X}%
}+t\mathbf{ad}_{\mathbf{U}}\right) ^{i}\right\vert _{t=0}  \notag \\
=& \left( \mathbf{ad}_{\mathbf{X}}+t\mathbf{ad}_{\mathbf{U}}\right) ^{i-1}%
\mathbf{ad}_{\mathbf{U}}+\left( \mathbf{ad}_{\mathbf{X}}+t\mathbf{ad}_{%
\mathbf{U}}\right) ^{i-2}\mathbf{ad}_{\mathbf{U}}\left( \mathbf{ad}_{\mathbf{%
X}}+t\mathbf{ad}_{\mathbf{U}}\right) +\ldots  \notag \\
& +\left. \left( \mathbf{ad}_{\mathbf{X}}+t\mathbf{ad}_{\mathbf{U}}\right) 
\mathbf{ad}_{\mathbf{U}}^{i-1}\right\vert _{t=0}=\sum_{j=0}^{i-1}\mathbf{P}%
_{j}\left( \mathbf{X}\right) \mathbf{ad}_{\mathbf{U}}\mathbf{P}%
_{i-j-1}\left( \mathbf{X}\right) .  \label{P}
\end{align}
\end{proof}

In particular, the following derivatives for $i=1,\ldots ,4$ will be needed%
\begin{eqnarray}
\left( \mathrm{D}_{\mathbf{X}}\mathbf{P}_{1}\right) 
\hspace{-0.5ex}%
\left( \mathbf{U}\right) &=&\mathbf{ad}_{\mathbf{U}},\ \ \left( \mathrm{D}_{%
\mathbf{X}}\mathbf{P}_{2}\right) 
\hspace{-0.5ex}%
\left( \mathbf{U}\right) =\mathbf{ad}_{\mathbf{X}}\mathbf{ad}_{\mathbf{U}}+%
\mathbf{ad}_{\mathbf{X}}\mathbf{ad}_{\mathbf{U}}  \notag \\
\left( \mathrm{D}_{\mathbf{X}}\mathbf{P}_{3}\right) 
\hspace{-0.5ex}%
\left( \mathbf{U}\right) &=&\mathbf{ad}_{\mathbf{X}}^{2}\mathbf{ad}_{\mathbf{%
U}}+\mathbf{ad}_{\mathbf{X}}\mathbf{ad}_{\mathbf{U}}\mathbf{ad}_{\mathbf{X}}+%
\mathbf{ad}_{\mathbf{X}}\mathbf{ad}_{\mathbf{U}}^{2}  \label{dPi} \\
\left( \mathrm{D}_{\mathbf{X}}\mathbf{P}_{4}\right) 
\hspace{-0.5ex}%
\left( \mathbf{U}\right) &=&\mathbf{ad}_{\mathbf{X}}^{3}\mathbf{ad}_{\mathbf{%
U}}+\mathbf{ad}_{\mathbf{X}}\mathbf{ad}_{\mathbf{U}}\mathbf{ad}_{\mathbf{X}%
}^{2}+\mathbf{ad}_{\mathbf{X}}^{2}\mathbf{ad}_{\mathbf{U}}\mathbf{ad}_{%
\mathbf{X}}+\mathbf{ad}_{\mathbf{X}}\mathbf{ad}_{\mathbf{U}}^{3}.  \notag
\end{eqnarray}

Also the directional derivatives of $\varphi =\left\Vert \mathbf{x}%
\right\Vert $ and of the parameters (\ref{alpha}) will be needed throughout
the paper. They are readily available in closed form \cite{RSPA2021}%
\begin{eqnarray}
\left( \mathrm{D}_{\mathbf{x}}\left\Vert \mathbf{x}\right\Vert \right)
\left( \mathbf{u}\right) &=&\tfrac{\mathbf{x}^{T}\mathbf{u}}{\varphi },(%
\mathrm{D}_{\mathbf{x}}\alpha )%
\hspace{-0.5ex}%
\left( \mathbf{u}\right) =\left( \delta -\tfrac{1}{2}\beta \right) \mathbf{x}%
^{T}\mathbf{u},(\mathrm{D}_{\mathbf{x}}\beta )%
\hspace{-0.5ex}%
\left( \mathbf{u}\right) =\tfrac{2\mathbf{x}^{T}\mathbf{u}}{\varphi ^{2}}%
\left( \alpha -\beta \right)  \notag \\
(\mathrm{D}_{\mathbf{x}}\delta )%
\hspace{-0.5ex}%
\left( \mathbf{u}\right) &=&\tfrac{\mathbf{x}^{T}\mathbf{u}}{\varphi ^{2}}%
\left( \tfrac{1}{2}\beta -3\delta \right) ,(\mathrm{D}_{\mathbf{x}}\gamma )%
\hspace{-0.5ex}%
\left( \mathbf{u}\right) =\left( \tfrac{\gamma ^{2}-\gamma }{\varphi ^{2}}-%
\tfrac{1}{4}\right) (\mathbf{x}^{T}\mathbf{u}).  \label{DD}
\end{eqnarray}

\subsection{First Directional Derivatives}

The directional derivatives of the matrices $\mathbf{dexp}$ and $\mathbf{dexp%
}^{-1}$ are derived from (\ref{dexpSE3}) and (\ref{dexpInvSE3}),
respectively. Using (\ref{dPi}) yields the following intermediate results

\begin{eqnarray}
\left( \mathrm{D}_{\mathbf{X}}\mathbf{dexp}\right) \left( \mathbf{U}\right)
&=&\sum_{i=1}^{4}a_{i}\left( \mathbf{X}\right) \left( \mathrm{D}_{\mathbf{X}}%
\mathbf{P}_{i}\right) \left( \mathbf{U}\right) +\sum_{i=1}^{4}\left( \mathrm{%
D}_{\mathbf{X}}a_{i}\right) \left( \mathbf{U}\right) \mathbf{P}_{i}\left( 
\mathbf{X}\right)  \notag \\
&=&a_{1}\mathbf{ad}_{\mathbf{U}}+a_{2}\left( \mathbf{ad}_{\mathbf{X}}\mathbf{%
ad}_{\mathbf{U}}+\mathbf{ad}_{\mathbf{U}}\mathbf{ad}_{\mathbf{X}}\right)
+a_{3}(\mathbf{ad}_{\mathbf{X}}^{2}\mathbf{ad}_{\mathbf{U}}+\mathbf{ad}_{%
\mathbf{U}}\mathbf{ad}_{\mathbf{X}}^{2}+\mathbf{ad}_{\mathbf{X}}\mathbf{ad}_{%
\mathbf{U}}\mathbf{ad}_{\mathbf{X}})  \notag \\
&&+a_{4}(\mathbf{ad}_{\mathbf{X}}^{3}\mathbf{ad}_{\mathbf{U}}+\mathbf{ad}_{%
\mathbf{X}}^{2}\mathbf{ad}_{\mathbf{U}}\mathbf{ad}_{\mathbf{X}}+\mathbf{ad}_{%
\mathbf{X}}\mathbf{ad}_{\mathbf{U}}\mathbf{ad}_{\mathbf{X}}^{2}+\mathbf{ad}_{%
\mathbf{U}}\mathbf{ad}_{\mathbf{X}}^{3})+\sum_{i=1}^{4}\left( \mathrm{D}_{%
\mathbf{X}}a_{i}\right) \left( \mathbf{U}\right) \mathbf{P}_{i}\left( 
\mathbf{X}\right)  \label{DdexpSE3} \\
\left( \mathrm{D}_{\mathbf{X}}\mathbf{dexp}^{-1}\right) \left( \mathbf{U}%
\right) &=&\sum_{i=1}^{4}b_{i}\left( \mathbf{X}\right) \left( \mathrm{D}_{%
\mathbf{X}}\mathbf{P}_{i}\right) \left( \mathbf{U}\right)
+\sum_{i=1}^{4}\left( \mathrm{D}_{\mathbf{X}}b_{i}\right) \left( \mathbf{U}%
\right) \mathbf{P}_{i}\left( \mathbf{X}\right)  \notag \\
&=&b_{1}\mathbf{ad}_{\mathbf{U}}+b_{2}\left( \mathbf{ad}_{\mathbf{X}}\mathbf{%
ad}_{\mathbf{U}}+\mathbf{ad}_{\mathbf{U}}\mathbf{ad}_{\mathbf{X}}\right) 
\notag \\
&&+b_{4}(\mathbf{ad}_{\mathbf{X}}^{3}\mathbf{ad}_{\mathbf{U}}+\mathbf{ad}_{%
\mathbf{X}}^{2}\mathbf{ad}_{\mathbf{U}}\mathbf{ad}_{\mathbf{X}}+\mathbf{ad}_{%
\mathbf{X}}\mathbf{ad}_{\mathbf{U}}\mathbf{ad}_{\mathbf{X}}^{2}+\mathbf{ad}_{%
\mathbf{U}}\mathbf{ad}_{\mathbf{X}}^{3})+\sum_{i=1}^{4}\left( \mathrm{D}_{%
\mathbf{X}}b_{i}\right) \left( \mathbf{U}\right) \mathbf{P}_{i}\left( 
\mathbf{X}\right) .  \label{DdexpInvSE3}
\end{eqnarray}

\begin{lemma}
The non-zero directional derivatives of the coefficients (\ref{ab}) are%
\begin{equation}
\left( \mathrm{D}_{\mathbf{X}}a_{i}\right) 
\hspace{-0.5ex}%
\left( \mathbf{U}\right) =\left( \mathbf{x}^{T}\mathbf{u}\right) \bar{a}%
_{i}\left( \mathbf{x}\right) ,\ \ \ \left( \mathrm{D}_{\mathbf{X}%
}b_{i}\right) 
\hspace{-0.5ex}%
\left( \mathbf{U}\right) =\left( \mathbf{x}^{T}\mathbf{u}\right) \bar{b}%
_{i}\left( \mathbf{x}\right)  \label{Da}
\end{equation}%
where $\bar{a}_{i},\bar{b}_{i}$ are functions of $\mathbf{x}$ only, with
limits for $\mathbf{x}\rightarrow \mathbf{0}$, as follows%
\begin{eqnarray}
\bar{a}_{1}\left( \mathbf{x}\right) &=&\tfrac{1}{4}\beta +\tfrac{1}{\varphi
^{2}}\left( \tfrac{5}{2}\alpha -2\beta -\tfrac{1}{2}\right) ,\ \ \bar{a}%
_{1}\left( \mathbf{0}\right) =0  \label{a1} \\
\bar{a}_{2}\left( \mathbf{x}\right) &=&\tfrac{1}{\varphi ^{2}}\tfrac{7}{4}%
\beta -\tfrac{1}{2}\alpha -\tfrac{15}{2}\delta =\tfrac{1}{\varphi ^{2}}%
\left( \tfrac{7}{4}\beta -\tfrac{1}{2}\alpha \right) +\tfrac{1}{\varphi ^{4}}%
\tfrac{15}{2}\left( \alpha -1\right) ,\ \ \bar{a}_{2}\left( \mathbf{0}%
\right) =0  \notag \\
\bar{a}_{3}\left( \mathbf{x}\right) &=&\tfrac{1}{\varphi ^{2}}\bar{a}%
_{1}\left( \mathbf{x}\right) ,\ \ \bar{a}_{3}\left( \mathbf{0}\right) =-%
\frac{1}{180}  \notag \\
\bar{a}_{4}\left( \mathbf{x}\right) &=&\tfrac{1}{\varphi ^{2}}\bar{a}%
_{2}\left( \mathbf{x}\right) ,\ \ \bar{a}_{4}\left( \mathbf{0}\right) =-%
\frac{1}{1260}  \notag \\
\bar{b}_{2}\left( \mathbf{x}\right) &=&\tfrac{1}{\varphi ^{2}}\left( \tfrac{3%
}{4}-\tfrac{3}{2}\tfrac{\delta }{\beta }\right) +\tfrac{1}{\varphi ^{4}}%
\left( \tfrac{\alpha +3\alpha ^{2}}{\beta ^{2}}-4\right) =\tfrac{1}{\varphi
^{2}}\tfrac{3}{4}+\tfrac{1}{\varphi ^{4}}\left( \gamma \left( \tfrac{3}{2}+%
\tfrac{1}{\beta }+3\gamma \right) -\tfrac{3}{2\beta }-4\right) ,\ \ \bar{b}%
_{2}\left( \mathbf{0}\right) =0  \label{b2} \\
\bar{b}_{4}\left( \mathbf{x}\right) &=&\tfrac{1}{\varphi ^{4}}\left( \tfrac{1%
}{4}-\tfrac{1}{2}\tfrac{\delta }{\beta }\right) +\tfrac{1}{\varphi ^{6}}%
\left( \left( 1+\alpha \right) \left( \tfrac{\alpha }{\beta ^{2}}+\tfrac{1}{%
\beta }\right) -4\right) =\tfrac{1}{4\varphi ^{4}}+\tfrac{1}{\varphi ^{6}}%
\left( \tfrac{1}{2\beta }+\tfrac{3}{2}\gamma +\tfrac{\gamma }{\beta }+\gamma
^{2}-4\right) ,\ \ \bar{b}_{4}\left( \mathbf{0}\right) =-\frac{1}{7560}. 
\notag
\end{eqnarray}
\end{lemma}

\begin{proof}
The result follows with the derivatives in (\ref{DD}) after algebraic
manipulation.
\end{proof}

Important for a robust implementation are the limits for $\mathbf{x}%
\rightarrow \mathbf{0}$. They are found as%
\begin{equation}
\left( \mathrm{D}_{\mathbf{0}}\mathbf{dexp}\right) \left( \mathbf{U}\right) =%
\frac{1}{2}\mathbf{ad}_{\mathbf{U}},\ \ \ \left( \mathrm{D}_{\mathbf{0}}%
\mathbf{dexp}^{-1}\right) \left( \mathbf{U}\right) =-\frac{1}{2}\mathbf{ad}_{%
\mathbf{U}}.
\end{equation}

\begin{remark}
The divisions by powers of $\varphi $, in the above expressions, may lead to
numerical issues for $\varphi \rightarrow 0$. This is mitigated by using the
normalized vectors $\mathbf{n}=\mathbf{x/}\varphi $ and $\mathbf{N}=\mathbf{%
X/}\varphi $. The expressions (\ref{DdexpSE3}) and (\ref{DdexpInvSE3}) are
replaced by their final form 
\begin{eqnarray}
\left( \mathrm{D}_{\mathbf{X}}\mathbf{dexp}\right) \left( \mathbf{U}\right)
&=&\left( \beta -\tfrac{\alpha }{2}\right) \mathbf{ad}_{\mathbf{U}}-\tfrac{%
\beta }{4}\left( \mathbf{ad}_{\mathbf{X}}\mathbf{ad}_{\mathbf{U}}+\mathbf{ad}%
_{\mathbf{U}}\mathbf{ad}_{\mathbf{X}}\right) +\tfrac{1}{\varphi }\left( 
\tfrac{5}{2}-\tfrac{5}{2}\alpha \right) \left( \mathbf{ad}_{\mathbf{N}}%
\mathbf{ad}_{\mathbf{U}}+\mathbf{ad}_{\mathbf{U}}\mathbf{ad}_{\mathbf{N}%
}\right)  \notag \\
&&+\tfrac{1}{2}\left( \beta -\alpha \right) (\mathbf{ad}_{\mathbf{N}}^{2}%
\mathbf{ad}_{\mathbf{U}}+\mathbf{ad}_{\mathbf{N}}\mathbf{ad}_{\mathbf{U}}%
\mathbf{ad}_{\mathbf{N}}+\mathbf{ad}_{\mathbf{U}}\mathbf{ad}_{\mathbf{N}%
}^{2})-\tfrac{\beta }{4}\mathbf{ad}_{\mathbf{X}}(\mathbf{ad}_{\mathbf{N}}^{2}%
\mathbf{ad}_{\mathbf{U}}+\mathbf{ad}_{\mathbf{U}}\mathbf{ad}_{\mathbf{N}%
}^{2})-\tfrac{\beta }{4}(\mathbf{ad}_{\mathbf{N}}^{2}\mathbf{ad}_{\mathbf{U}%
}+\mathbf{ad}_{\mathbf{U}}\mathbf{ad}_{\mathbf{N}}^{2})\mathbf{ad}_{\mathbf{X%
}}  \notag \\
&&+\tfrac{1}{\varphi }\left( \tfrac{3}{2}-\tfrac{3}{2}\alpha \right) (%
\mathbf{ad}_{\mathbf{N}}^{3}\mathbf{ad}_{\mathbf{U}}+\mathbf{ad}_{\mathbf{N}%
}^{2}\mathbf{ad}_{\mathbf{U}}\mathbf{ad}_{\mathbf{N}}+\mathbf{ad}_{\mathbf{N}%
}\mathbf{ad}_{\mathbf{U}}\mathbf{ad}_{\mathbf{N}}^{2}+\mathbf{ad}_{\mathbf{U}%
}\mathbf{ad}_{\mathbf{N}}^{3})+\tfrac{\beta }{4}(\mathbf{x}^{T}\mathbf{u})%
\mathbf{ad}_{\mathbf{X}}+\left( \tfrac{5}{2}\alpha -2\beta -\tfrac{1}{2}%
\right) (\mathbf{n}^{T}\mathbf{u})\mathbf{ad}_{\mathbf{N}}  \notag \\
&&+\left[ \left( \tfrac{7}{4}\beta -\tfrac{1}{2}\alpha \right) (\mathbf{x}%
^{T}\mathbf{u})+\tfrac{1}{\varphi }\tfrac{15}{2}\left( \alpha -1\right)
\left( \mathbf{n}^{T}\mathbf{u}\right) \right] \mathbf{ad}_{\mathbf{N}}^{2}+%
\tfrac{\beta }{4}(\mathbf{n}^{T}\mathbf{u})\mathbf{ad}_{\mathbf{X}}^{2}%
\mathbf{ad}_{\mathbf{N}}  \notag \\
&&+\left( \tfrac{5}{2}\alpha -2\beta -\tfrac{1}{2}\right) (\mathbf{n}^{T}%
\mathbf{u})\mathbf{ad}_{\mathbf{N}}^{3}+\left[ \left( \tfrac{7}{4}\beta -%
\tfrac{1}{2}\alpha \right) (\mathbf{x}^{T}\mathbf{u})+\tfrac{1}{\varphi }%
\tfrac{15}{2}\left( \alpha -1\right) \left( \mathbf{n}^{T}\mathbf{u}\right) %
\right] \mathbf{ad}_{\mathbf{N}}^{4}  \label{dexpAlt}
\end{eqnarray}%
\begin{eqnarray}
\left( \mathrm{D}_{\mathbf{X}}\mathbf{dexp}^{-1}\right) \left( \mathbf{U}%
\right) &=&-\tfrac{1}{2}\mathbf{ad}_{\mathbf{U}}+\tfrac{1}{\varphi }\left( 2-%
\tfrac{1+3\alpha }{2\beta }\right) \left( \mathbf{ad}_{\mathbf{N}}\mathbf{ad}%
_{\mathbf{U}}+\mathbf{ad}_{\mathbf{U}}\mathbf{ad}_{\mathbf{N}}\right)  \notag
\\
&&+\tfrac{1}{\varphi }\left( 1-\tfrac{1+\alpha }{2\beta }\right) (\mathbf{ad}%
_{\mathbf{N}}^{3}\mathbf{ad}_{\mathbf{U}}+\mathbf{ad}_{\mathbf{N}}^{2}%
\mathbf{ad}_{\mathbf{U}}\mathbf{ad}_{\mathbf{N}}+\mathbf{ad}_{\mathbf{N}}%
\mathbf{ad}_{\mathbf{U}}\mathbf{ad}_{\mathbf{N}}^{2}+\mathbf{ad}_{\mathbf{U}}%
\mathbf{ad}_{\mathbf{N}}^{3})  \notag \\
&&+\left( \tfrac{3}{4}\mathbf{x}^{T}\mathbf{u}+\tfrac{1}{\varphi }\left(
\gamma \left( \tfrac{3}{2}+\tfrac{1}{\beta }+3\gamma \right) -\tfrac{3}{%
2\beta }-4\right) \mathbf{n}^{T}\mathbf{u}\right) \mathbf{ad}_{\mathbf{N}%
}^{2}  \label{dexpInvAlt} \\
&&+\left( \tfrac{1}{4}\mathbf{x}^{T}\mathbf{u}+\tfrac{1}{\varphi }\left( 
\tfrac{1}{2\beta }+\tfrac{3}{2}\gamma +\tfrac{\gamma }{\beta }+\gamma
^{2}-4\right) \mathbf{n}^{T}\mathbf{u}\right) \mathbf{ad}_{\mathbf{N}}^{4}. 
\notag
\end{eqnarray}
\end{remark}

\subsection{Jacobian of the Evaluation Map of \textbf{dexp}}

In many applications, the partial derivatives w.r.t. the canonical
coordinates $\mathbf{X}$ of the vector defined by the expression $f\left( 
\mathbf{X}\right) =\mathbf{dexp}_{\mathbf{X}}\mathbf{Z}$ are needed, for
given $\mathbf{Z}\in {\mathbb{R}}^{3}\times {\mathbb{R}}^{3}$. This is the
derivative of the dexp map when applied to $\mathbf{Z}$, i.e. evaluated with 
$\mathbf{Z}$. Therefore, $f\left( \mathbf{X}\right) $ will be referred to as
the \emph{evaluation map}. With slight abuse of language, the term $\frac{%
\partial f}{\partial \mathbf{X}}$ will be referred to as the Jacobian.

\begin{lemma}
The Jacobian of $f\left( \mathbf{X}\right) =\mathbf{dexp}_{\mathbf{X}}%
\mathbf{Z}$ and $f^{\mathrm{inv}}\left( \mathbf{X}\right) =\mathbf{dexp}_{%
\mathbf{X}}^{-1}\mathbf{Z}$ admit the following closed form expressions,
abbreviating again $\mathbf{Z}_{i}:=\mathbf{ad}_{\mathbf{X}}^{i}\mathbf{Z}$, 
\begin{eqnarray}
\frac{\partial f}{\partial \mathbf{X}} &=&-a_{1}\mathbf{ad}_{\mathbf{Z}%
}+a_{2}\left( \mathbf{ad}_{\mathbf{Z}}\mathbf{ad}_{\mathbf{X}}-2\mathbf{ad}_{%
\mathbf{X}}\mathbf{ad}_{\mathbf{Z}}\right) -a_{3}\left( \mathbf{ad}_{\mathbf{%
X}}^{2}\mathbf{ad}_{\mathbf{Z}}+\mathbf{ad}_{\mathbf{X}}\mathbf{ad}_{\mathbf{%
Z}_{1}}+\mathbf{ad}_{\mathbf{Z}_{2}}\right)  \notag \\
&&-a_{4}\left( \mathbf{ad}_{\mathbf{X}}^{3}\mathbf{ad}_{\mathbf{Z}}+\mathbf{%
ad}_{\mathbf{X}}\mathbf{ad}_{\mathbf{Z}_{2}}+\mathbf{ad}_{\mathbf{x}}^{2}%
\mathbf{ad}_{\mathbf{Z}_{1}}+\mathbf{ad}_{\mathbf{Z}_{3}}\right)
+\sum_{i=1}^{4}\bar{a}_{i}\mathbf{Z}_{i}\cdot 
\begin{bmatrix}
\mathbf{x}^{T} & \mathbf{0}%
\end{bmatrix}
\label{PartialDexp} \\
\frac{\partial f^{\mathrm{inv}}}{\partial \mathbf{X}} &=&-b_{1}\mathbf{ad}_{%
\mathbf{Z}}+b_{2}\left( \mathbf{ad}_{\mathbf{Z}}\mathbf{ad}_{\mathbf{X}}-2%
\mathbf{ad}_{\mathbf{X}}\mathbf{ad}_{\mathbf{Z}}\right)  \notag \\
&&-b_{4}\left( \mathbf{ad}_{\mathbf{X}}^{3}\mathbf{ad}_{\mathbf{Z}}+\mathbf{%
ad}_{\mathbf{X}}\mathbf{ad}_{\mathbf{Z}_{2}}+\mathbf{ad}_{\mathbf{x}}^{2}%
\mathbf{ad}_{\mathbf{Z}_{1}}+\mathbf{ad}_{\mathbf{Z}_{3}}\right)
+\sum_{i=1}^{4}\bar{b}_{i}\mathbf{Z}_{i}\cdot 
\begin{bmatrix}
\mathbf{x}^{T} & \mathbf{0}%
\end{bmatrix}%
\end{eqnarray}%
The limits for $\mathbf{x}\rightarrow \mathbf{0}$ are $\frac{\partial f}{%
\partial \mathbf{X}}\left( \mathbf{0}\right) =-\frac{1}{2}\mathbf{ad}_{%
\mathbf{Z}}$ and $\frac{\partial f^{\mathrm{inv}}}{\partial \mathbf{X}}%
\left( \mathbf{0}\right) =\frac{1}{2}\mathbf{ad}_{\mathbf{Z}}$.
\end{lemma}

\begin{proof}
Denoting with $U_{k}$ the components of $\mathbf{U}$, the directional
derivative of $\mathbf{dexp}_{\mathbf{X}}\mathbf{Z}$ along $\mathbf{U}$ is $%
\left( \mathrm{D}_{\mathbf{X}}\mathbf{dexp}\cdot \mathbf{Z}\right) (\mathbf{U%
})=\left( \mathrm{D}_{\mathbf{X}}\mathbf{dexp}\right) (\mathbf{U})\cdot 
\mathbf{Z}=\sum_{k}\frac{\partial }{\partial X_{k}}(\mathbf{dexp}_{\mathbf{X}%
}\mathbf{Z)}U_{k}$. The coefficient matrix $\frac{\partial }{\partial 
\mathbf{X}}(\mathbf{dexp}_{\mathbf{X}}\mathbf{Z)}$ is thus obtained when
written explicitly as multiplication with vector $\mathbf{U}$. When (\ref%
{DdexpSE3}) is multiplied with $\mathbf{Z}$, the term multiplied with $a_{1}$
is written $\mathbf{ad}_{\mathbf{U}}\mathbf{Z}=-\mathbf{ad}_{\mathbf{Z}}%
\mathbf{U}$ using the skew symmetry of the Lie bracket. The term with $a_{2}$
is $\mathbf{ad}_{\mathbf{X}}\mathbf{ad}_{\mathbf{U}}\mathbf{Z}+\mathbf{ad}_{%
\mathbf{U}}\mathbf{ad}_{\mathbf{X}}\mathbf{Z}$. Applying the skew symmetry
to the first term yields $\mathbf{ad}_{\mathbf{X}}\mathbf{ad}_{\mathbf{U}}%
\mathbf{Z=-ad}_{\mathbf{X}}\mathbf{ad}_{\mathbf{U}}\mathbf{U}$, and the
Jacobi identity to the second term yields $\mathbf{ad}_{\mathbf{U}}\mathbf{ad%
}_{\mathbf{X}}\mathbf{Z}=-\mathbf{ad}_{\mathbf{X}}\mathbf{ad}_{\mathbf{Z}}%
\mathbf{U}-\mathbf{ad}_{\mathbf{Z}}\mathbf{ad}_{\mathbf{U}}\mathbf{X}$. The
skew symmetry applied to the last terms thus yields $-2\mathbf{ad}_{\mathbf{X%
}}\mathbf{ad}_{\mathbf{Z}}\mathbf{U}-\mathbf{ad}_{\mathbf{Z}}\mathbf{ad}_{%
\mathbf{X}}\mathbf{U}$ for the term multiplied with $a_{2}$. The term
multiplied with $a_{3}$ is transformed using the skew symmetry as $\mathbf{ad%
}_{\mathbf{X}}^{2}\mathbf{ad}_{\mathbf{U}}\mathbf{Z}+\mathbf{ad}_{\mathbf{U}}%
\mathbf{ad}_{\mathbf{X}}^{2}\mathbf{Z}+\mathbf{ad}_{\mathbf{X}}\mathbf{ad}_{%
\mathbf{U}}\mathbf{ad}_{\mathbf{X}}\mathbf{Z}=-\mathbf{ad}_{\mathbf{X}}^{2}%
\mathbf{ad}_{\mathbf{Z}}\mathbf{U}-\mathbf{ad}_{\mathbf{ad}_{\mathbf{X}}^{2}}%
\mathbf{U}-\mathbf{ad}_{\mathbf{X}}\mathbf{ad}_{\mathbf{ad}_{\mathbf{X}}%
\mathbf{Z}}\mathbf{U}$. The term multiplied with $a_{4}$ is treated
analogously. The last term in (\ref{PartialDexp}) follows from (\ref%
{DdexpSE3}) and writing the differential of $a_{i}$, using (\ref{Da}), as $%
\left( \mathrm{D}_{\mathbf{X}}a_{i}\right) \left( \mathbf{U}\right) \mathbf{P%
}_{i}\mathbf{Z}=\bar{a}_{i}\mathbf{P}_{i}\mathbf{Z}\left( \mathbf{x}^{T}%
\mathbf{u}\right) =\bar{a}_{i}\mathbf{P}_{i}\mathbf{Zx}^{T}\mathbf{u}$, with
notion (\ref{Pi}). Since $\bar{a}_{i}$ depend on $\mathbf{x}$ only, the last
three components of the derivatives are zero, hence the vector $%
\begin{bmatrix}
\mathbf{x}^{T} & \mathbf{0}%
\end{bmatrix}%
$. The derivation for the inverse proceeds in the same way.
\end{proof}

\subsection{Jacobian of the Evaluation Map of $\mathbf{dexp}^{T}$}

Consider the map defined as $\bar{f}\left( \mathbf{X}\right) =\mathbf{dexp}_{%
\mathbf{X}}^{T}\mathbf{Z}$, for $\mathbf{Z}\in {\mathbb{R}}^{3}\times {%
\mathbb{R}}^{3}$. The Jacobian $\frac{\partial \bar{f}}{\partial \mathbf{X}}$
of this map is also relevant in different contexts. 
\color[rgb]{0,0,0}%
Introduce the matrix%
\begin{equation}
\overline{\mathbf{ad}}_{\mathbf{U}}:=%
\begin{bmatrix}
\tilde{\mathbf{u}} & \tilde{\mathbf{v}} \\ 
\tilde{\mathbf{v}} & \mathbf{0}%
\end{bmatrix}%
,\ \mathrm{with}\ \mathbf{U}=\left( \mathbf{u},\mathbf{v}\right)
\label{adBar}
\end{equation}%
so that $\mathbf{ad}_{\mathbf{X}}^{T}\mathbf{U}=\overline{\mathbf{ad}}_{%
\mathbf{U}}\mathbf{X}$. This is always possible as the coadjoint action $%
\mathbf{Q}\in se^{\ast }\left( 3\right) \mapsto \mathbf{ad}_{\mathbf{X}}^{T}%
\mathbf{Q}\in se^{\ast }\left( 3\right) $, with $\mathbf{X}\in se\left(
3\right) $, is bilinear in $\mathbf{X}$ and $\mathbf{Q}$. The so defined
matrix is skew symmetric, i.e. $\overline{\mathbf{ad}}_{\mathbf{X}}^{T}=-%
\overline{\mathbf{ad}}_{\mathbf{X}}$.%
\color{black}%

\begin{lemma}
The Jacobian of $\bar{f}\left( \mathbf{X}\right) =\mathbf{dexp}_{\mathbf{X}%
}^{T}\mathbf{Z}$ admits the following closed form expressions, abbreviating
again $\mathbf{Z}_{i}:=\mathbf{ad}_{\mathbf{X}}^{i}\mathbf{Z},\overline{%
\mathbf{Q}}_{i}:=\mathbf{P}_{i}^{T}\mathbf{Q}$, and introducing $\overline{%
\mathbf{Z}}_{i}:=\mathbf{P}_{i}^{T}\mathbf{Z}$,%
\begin{equation}
\frac{\partial \bar{f}}{\partial \mathbf{X}}=a_{1}\overline{\mathbf{ad}}_{%
\mathbf{Z}}+a_{2}\left( \overline{\mathbf{ad}}_{\overline{\mathbf{Z}}_{1}}+%
\mathbf{P}_{1}^{T}\overline{\mathbf{ad}}_{\mathbf{Z}}\right) +a_{3}\left( 
\mathbf{P}_{2}^{T}\overline{\mathbf{ad}}_{\mathbf{Z}}+\mathbf{P}_{1}^{T}%
\overline{\mathbf{ad}}_{\overline{\mathbf{Z}}_{1}}+\overline{\mathbf{ad}}_{%
\overline{\mathbf{Z}}_{2}}\right) +a_{4}\left( \mathbf{P}_{3}^{T}\overline{%
\mathbf{ad}}_{\mathbf{Z}}+\mathbf{P}_{1}^{T}\overline{\mathbf{ad}}_{%
\overline{\mathbf{Z}}_{2}}+\mathbf{P}_{2}^{T}\overline{\mathbf{ad}}_{%
\overline{\mathbf{Z}}_{1}}+\overline{\mathbf{ad}}_{\overline{\mathbf{Z}}%
_{3}}\right) +\sum_{i=1}^{4}\bar{a}_{i}\overline{\mathbf{Z}}_{i}\cdot 
\begin{bmatrix}
\mathbf{x}^{T} & \mathbf{0}%
\end{bmatrix}%
\end{equation}%
The limits for $\mathbf{x}\rightarrow \mathbf{0}$ is $\frac{\partial \bar{f}%
}{\partial \mathbf{X}}\left( \mathbf{0}\right) =\frac{1}{2}\overline{\mathbf{%
ad}}_{\mathbf{Z}}$.
\end{lemma}

\begin{proof}
The directional derivative $\left( \mathrm{D}_{\mathbf{X}}\mathbf{dexp}^{T}%
\mathbf{Z}\right) (\mathbf{U})$ is expressed with (\ref{DdexpSE3}) as%
\begin{eqnarray*}
\left( \mathrm{D}_{\mathbf{X}}\mathbf{dexp}^{T}\mathbf{Z}\right) (\mathbf{U}%
) &=&a_{1}\mathbf{ad}_{\mathbf{U}}^{T}\mathbf{Z}+a_{2}\left( \mathbf{P}^{T}%
\mathbf{ad}_{\mathbf{U}}^{T}+\mathbf{ad}_{\mathbf{U}}^{T}\mathbf{P}%
^{T}\right) \mathbf{Z}+a_{3}(\mathbf{P}_{2}^{T}\mathbf{ad}_{\mathbf{U}}^{T}+%
\mathbf{ad}_{\mathbf{U}}^{T}\mathbf{P}_{2}^{T}+\mathbf{P}^{T}\mathbf{ad}_{%
\mathbf{U}}^{T}\mathbf{P}^{T})\mathbf{Z} \\
&&+a_{4}(\mathbf{P}_{3}^{T}\mathbf{ad}_{\mathbf{U}}^{T}+\mathbf{P}_{2}^{T}%
\mathbf{ad}_{\mathbf{U}}^{T}\mathbf{P}^{T}+\mathbf{P}^{T}\mathbf{ad}_{%
\mathbf{U}}^{T}\mathbf{P}_{2}^{T}+\mathbf{ad}_{\mathbf{U}}^{T}\mathbf{P}%
_{3}^{T})\mathbf{Z}+\sum_{i=1}^{4}\left( \mathrm{D}_{\mathbf{X}}a_{i}\right)
\left( \mathbf{U}\right) \mathbf{P}_{i}^{T}\left( \mathbf{X}\right) \mathbf{Z%
}.
\end{eqnarray*}%
Substituting relations $\mathbf{ad}_{\mathbf{U}}^{T}\mathbf{Z}=\overline{%
\mathbf{ad}}_{\mathbf{Z}}\mathbf{U}$ etc. yields $\left( \mathrm{D}_{\mathbf{%
X}}\mathbf{dexp}^{T}\mathbf{Z}\right) (\mathbf{U})=\frac{\partial \bar{f}}{%
\partial \mathbf{X}}\mathbf{U}$, and the statement follows.
\end{proof}

\subsection{Second Directional Derivatives of $\mathbf{dexp}_{\mathbf{X}}$
and $\mathbf{dexp}_{\mathbf{X}}^{-1}$}

With (\ref{dexpSE3}), the second directional derivative of $\mathbf{dexp}_{%
\mathbf{X}}$ along $\mathbf{U}=\left( \mathbf{u},\mathbf{v}\right) $ and $%
\mathbf{S}=\left( \mathbf{s},\mathbf{r}\right) $ is 
\begin{eqnarray}
\left( \mathrm{D}_{\mathbf{X}}^{2}\mathbf{dexp}\right) \left( \mathbf{U}%
\right) \left( \mathbf{S}\right) &=&\sum_{i=1}^{4}\left( \mathrm{D}_{\mathbf{%
X}}^{2}a_{i}\right) \left( \mathbf{U}\right) \left( \mathbf{S}\right) 
\mathbf{P}_{i}\left( \mathbf{X}\right) +\sum_{i=1}^{4}a_{i}\left( \mathbf{X}%
\right) \left( \mathrm{D}_{\mathbf{X}}^{2}\mathbf{P}_{i}\right) \left( 
\mathbf{U}\right) \left( \mathbf{S}\right)  \notag \\
&&+\sum_{i=1}^{4}\left( \mathrm{D}_{\mathbf{X}}a_{i}\right) \left( \mathbf{U}%
\right) \left( \mathrm{D}_{\mathbf{X}}\mathbf{P}_{i}\right) \left( \mathbf{S}%
\right) +\sum_{i=1}^{4}\left( \mathrm{D}_{\mathbf{X}}a_{i}\right) \left( 
\mathbf{S}\right) \left( \mathrm{D}_{\mathbf{X}}\mathbf{P}_{i}\right) \left( 
\mathbf{U}\right)  \label{d2exp}
\end{eqnarray}%
and the second derivative of $\mathbf{dexp}_{\mathbf{X}}^{-1}$ is obtained
by replacing $a_{i}$ with $b_{i}$ in (\ref{d2exp}).

\begin{lemma}
The second directional derivatives of $\mathbf{P}_{i}$ along $\mathbf{U}%
=\left( \mathbf{u},\mathbf{v}\right) $ and $\mathbf{S}=\left( \mathbf{s},%
\mathbf{r}\right) $ can be written as%
\begin{equation}
(\mathrm{D}_{\mathbf{X}}^{2}\mathbf{P}_{i})%
\hspace{-0.5ex}%
\left( \mathbf{U}\right) (\mathbf{S})=\sum_{j=1}^{i-1}\left( \left( \mathrm{D%
}_{\mathbf{X}}\mathbf{P}_{j}\right) 
\hspace{-0.5ex}%
\left( \mathbf{S}\right) \mathbf{ad}_{\mathbf{U}}\mathbf{P}_{i-j-1}(\mathbf{X%
})+\mathbf{P}_{i-j-1}(\mathbf{X})\mathbf{ad}_{\mathbf{U}}\left( \mathrm{D}_{%
\mathbf{X}}\mathbf{P}_{j}\right) 
\hspace{-0.5ex}%
\left( \mathbf{S}\right) \right) ,i\geq 2.  \label{d2PiGeneral}
\end{equation}%
The non-zero second directional derivatives of the coefficients (\ref{ab})
are%
\begin{eqnarray}
\left( \mathrm{D}_{\mathbf{X}}^{2}a_{i}\right) 
\hspace{-0.5ex}%
\left( \mathbf{U}\right) 
\hspace{-0.5ex}%
\left( \mathbf{S}\right) &=&\left( \mathbf{s}^{T}\mathbf{u}\right) \bar{a}%
_{i}\left( \mathbf{x}\right) +\left( \mathbf{x}^{T}\mathbf{u}\right) \left( 
\mathbf{x}^{T}\mathbf{s}\right) \breve{a}_{i}\left( \mathbf{x}\right)
\label{D2ai} \\
\left( \mathrm{D}_{\mathbf{X}}^{2}b_{i}\right) 
\hspace{-0.5ex}%
\left( \mathbf{U}\right) 
\hspace{-0.5ex}%
\left( \mathbf{S}\right) &=&\left( \mathbf{s}^{T}\mathbf{u}\right) \bar{b}%
_{i}\left( \mathbf{x}\right) +\left( \mathbf{x}^{T}\mathbf{u}\right) \left( 
\mathbf{x}^{T}\mathbf{s}\right) \breve{b}_{i}\left( \mathbf{x}\right)
\label{D2bi}
\end{eqnarray}%
with (expressions in terms of parameter $\delta $ are omitted)%
\begin{eqnarray}
\breve{a}_{1}\left( \mathbf{x}\right) &=&\tfrac{1}{\varphi ^{2}}\left( 
\tfrac{\alpha }{2}-\tfrac{7}{4}\beta \right) +\tfrac{1}{\varphi ^{4}}\left( 
\tfrac{7}{2}-\tfrac{23}{2}\alpha +8\beta \right) ,\ \breve{a}_{1}\left( 
\mathbf{0}\right) =-\frac{1}{90}  \label{aai} \\
\breve{a}_{2}\left( \mathbf{x}\right) &=&\tfrac{1}{\varphi ^{2}}\tfrac{1}{4}%
\beta +\tfrac{1}{\varphi ^{4}}\left( 5\alpha -\tfrac{43}{4}\beta -\tfrac{1}{2%
}\right) +\tfrac{1}{\varphi ^{6}}\left( \tfrac{75}{2}-\tfrac{75}{2}\alpha
\right) ,\ \breve{a}_{0}\left( \mathbf{0}\right) =-\frac{1}{630}  \notag \\
\breve{a}_{3}\left( \mathbf{x}\right) &=&\tfrac{1}{\varphi ^{4}}\left( 
\tfrac{1}{2}\alpha -\tfrac{9}{4}\beta \right) +\tfrac{1}{\varphi ^{6}}\left( 
\tfrac{9}{2}-\tfrac{33}{2}\alpha +12\beta \right) ,\ \breve{a}_{3}\left( 
\mathbf{0}\right) =-\frac{1}{1680}  \notag \\
\breve{a}_{4}\left( \mathbf{x}\right) &=&\tfrac{1}{\varphi ^{4}}\tfrac{1}{4}%
\beta +\tfrac{1}{\varphi ^{6}}\left( 6\alpha -\tfrac{1}{2}-\tfrac{57}{4}%
\beta \right) +\tfrac{1}{\varphi ^{8}}\left( \tfrac{105}{2}-\tfrac{105}{2}%
\alpha \right) ,\ \breve{a}_{4}\left( \mathbf{0}\right) =\frac{1}{15120} 
\notag
\end{eqnarray}%
\begin{eqnarray}
\breve{b}_{2}\left( \mathbf{x}\right) &=&\tfrac{1}{\varphi ^{6}}\left( 16+%
\tfrac{1}{\beta ^{2}}+\left( \tfrac{8}{\beta }-\tfrac{9}{2}\right) \gamma
-\left( 9+\tfrac{4}{\beta }\right) \gamma ^{2}-12\gamma ^{3}+\tfrac{9}{%
2\beta }\right) -\tfrac{1}{\varphi ^{4}}\left( \tfrac{9}{4}+3\gamma +\tfrac{1%
}{2\beta }\right) ,\ \breve{b}_{2}\left( \mathbf{0}\right) =-\frac{1}{3780}
\label{bbi} \\
\breve{b}_{4}\left( \mathbf{x}\right) &=&\tfrac{1}{\varphi ^{8}}\left( 24-%
\tfrac{2}{\beta }-\left( \tfrac{15}{2}+\tfrac{4}{\beta }\right) \gamma
-\left( \tfrac{11}{2}+\tfrac{3}{\beta }\right) \gamma ^{2}-2\gamma
^{3}\right) -\tfrac{1}{\varphi ^{6}}\left( \tfrac{11}{8}+\tfrac{1}{4\beta }+%
\tfrac{1}{2}\gamma \right) ,\ \breve{b}_{4}\left( \mathbf{0}\right) =-\frac{1%
}{50400}.  \notag
\end{eqnarray}
\end{lemma}

\begin{proof}
Relation (\ref{d2PiGeneral}) follows from (\ref{DPi}). According to (\ref{DD}%
), the derivative of all $\bar{a}_{i}$ involve the term $\mathbf{x}^{T}%
\mathbf{s}$, which are thus of the form $\left( \mathrm{D}_{\mathbf{X}}\bar{a%
}_{i}\right) 
\hspace{-0.5ex}%
\left( \mathbf{s}\right) =\left( \mathbf{x}^{T}\mathbf{s}\right) \breve{a}%
_{i}$, and the same holds true for $\bar{b}_{i}$. The explicit expressions
for $\breve{a}_{i}\left( \mathbf{x}\right) $ and $\breve{b}_{i}\left( 
\mathbf{x}\right) $ are obtained after application of (\ref{DD}) to (\ref{a1}%
), and (\ref{b2}), respectively.
\end{proof}

For $i=1,\ldots ,4$, relation (\ref{d2PiGeneral}) yields $(\mathrm{D}_{\hat{%
\mathbf{X}}}^{1}\mathbf{P}_{1})%
\hspace{-0.5ex}%
\left( \mathbf{U}\right) 
\hspace{-0.5ex}%
\left( \mathbf{S}\right) =\mathbf{0}$, and%
\begin{eqnarray}
\left( \mathrm{D}_{\mathbf{X}}^{2}\mathbf{P}_{2}\right) 
\hspace{-0.5ex}%
\left( \mathbf{U}\right) 
\hspace{-0.5ex}%
\left( \mathbf{S}\right) &=&\mathbf{ad}_{\mathbf{S}}\mathbf{ad}_{\mathbf{U}}+%
\mathbf{ad}_{\mathbf{U}}\mathbf{ad}_{\mathbf{S}}  \notag \\
\left( \mathrm{D}_{\mathbf{X}}^{2}\mathbf{P}_{3}\right) 
\hspace{-0.5ex}%
\left( \mathbf{U}\right) 
\hspace{-0.5ex}%
\left( \mathbf{S}\right) &=&(\mathbf{ad}_{\mathbf{X}}\mathbf{ad}_{\mathbf{S}%
}+\mathbf{ad}_{\mathbf{S}}\mathbf{ad}_{\mathbf{X}})\mathbf{ad}_{\mathbf{U}}+(%
\mathbf{ad}_{\mathbf{U}}\mathbf{ad}_{\mathbf{S}}+\mathbf{ad}_{\mathbf{S}}%
\mathbf{ad}_{\mathbf{U}})\mathbf{ad}_{\mathbf{X}}+(\mathbf{ad}_{\mathbf{U}}%
\mathbf{ad}_{\mathbf{X}}+\mathbf{ad}_{\mathbf{X}}\mathbf{ad}_{\mathbf{U}})%
\mathbf{ad}_{\mathbf{S}}  \notag \\
\left( \mathrm{D}_{\mathbf{X}}^{2}\mathbf{P}_{4}\right) 
\hspace{-0.5ex}%
\left( \mathbf{U}\right) 
\hspace{-0.5ex}%
\left( \mathbf{S}\right) &=&\mathbf{ad}_{\mathbf{S}}(\mathbf{ad}_{\mathbf{X}%
}^{2}\mathbf{ad}_{\mathbf{U}}+\mathbf{ad}_{\mathbf{U}}\mathbf{ad}_{\mathbf{X}%
}^{2}+\mathbf{ad}_{\mathbf{X}}\mathbf{ad}_{\mathbf{U}}\mathbf{ad}_{\mathbf{X}%
})+\mathbf{ad}_{\mathbf{U}}\left( \mathbf{ad}_{\mathbf{X}}^{2}\mathbf{ad}_{%
\mathbf{S}}+\mathbf{ad}_{\mathbf{S}}\mathbf{ad}_{\mathbf{X}}^{2}+\mathbf{ad}%
_{\mathbf{X}}\mathbf{ad}_{\mathbf{S}}\mathbf{ad}_{\mathbf{X}}\right)  \notag
\\
&&+\mathbf{ad}_{\mathbf{X}}^{2}(\mathbf{ad}_{\mathbf{S}}\mathbf{ad}_{\mathbf{%
U}}+\mathbf{ad}_{\mathbf{U}}\mathbf{ad}_{\mathbf{S}})+\mathbf{ad}_{\mathbf{X}%
}(\mathbf{ad}_{\mathbf{S}}\left( \mathbf{ad}_{\mathbf{X}}\mathbf{ad}_{%
\mathbf{U}}+\mathbf{ad}_{\mathbf{U}}\mathbf{ad}_{\mathbf{X}}\right) +\mathbf{%
ad}_{\mathbf{U}}\left( \mathbf{ad}_{\mathbf{X}}\mathbf{ad}_{\mathbf{S}}+%
\mathbf{ad}_{\mathbf{S}}\mathbf{ad}_{\mathbf{X}}\right) ).  \label{d2Pi}
\end{eqnarray}

\begin{remark}
The power of $\varphi $ appearing in the denominators in (\ref{aai}) and (%
\ref{bbi}) can be reduced expressing the final relation using the normalized
vector $\mathbf{n}:=\mathbf{x}/\left\Vert \mathbf{x}\right\Vert $ and $%
\mathbf{N}:=\mathbf{X}/\left\Vert \mathbf{x}\right\Vert $, analogously to (%
\ref{dexpAlt}) and (\ref{dexpInvAlt}). The explicit final expressions are
readily obtained, and are omitted here.
\end{remark}

\subsection{A Hessian Matrix}

Consider the expression $h\left( \mathbf{X}\right) =\mathbf{Q}^{T}\mathbf{%
dexp}_{\mathbf{X}}\mathbf{Z}$, which, for given $\mathbf{Q,Z}\in \mathbb{R}%
^{3}\times \mathbb{R}^{3}$, is a function of $\mathbf{X}$ (more precisely,
the arguments should be interpreted as $\mathbf{Q}\in se^{\ast }\left(
3\right) $ and $\mathbf{Z}\in se\left( 3\right) $). The second directional
derivative can be expressed as $\left( \mathrm{D}_{\mathbf{X}}^{2}\left( 
\mathbf{Q}^{T}\mathbf{dexp}\,\mathbf{Z}\right) \right) (\mathbf{U})(\mathbf{S%
})=\left( \mathrm{D}_{\mathbf{X}}^{2}\mathbf{dexp}\right) (\mathbf{U})(%
\mathbf{S})\mathbf{Z}=\sum_{k,l}\frac{\partial ^{2}h}{\partial X_{k}\partial
X_{l}}S_{k}U_{l}=\mathbf{S}^{T}\mathbf{HU}$, where $\mathbf{H}=\left[ \frac{%
\partial ^{2}h}{\partial X_{i}\partial X_{j}}\right] $ is the Hessian of $h$.

\begin{lemma}
The Hessian of the function $h\left( \mathbf{X}\right) =\mathbf{Q}^{T}%
\mathbf{dexp}_{\mathbf{X}}\mathbf{Z}$ is $\mathbf{H}=\mathbf{H}_{1}+\mathbf{H%
}_{2}+\mathbf{H}_{3}$, with%
\begin{equation}
\mathbf{H}_{1}=\sum_{i=1}^{4}\mathbf{H}_{1i},\ \mathbf{H}_{2}=\sum_{i=2}^{4}%
\left( \mathbf{H}_{2i}+\mathbf{H}_{2i}^{T}\right) ,\ \ \mathbf{H}%
_{3}=\sum_{i=1}^{4}\left( \mathbf{H}_{3i}+\mathbf{H}_{3i}^{T}\right) 
\label{H123}
\end{equation}%
where the individual matrices posses the following closed forms, denoting $%
\mathbf{Z}_{i}:=\mathbf{ad}_{\mathbf{X}}^{i}\mathbf{Z}$ and $\overline{%
\mathbf{Q}}_{i}:=(\mathbf{ad}_{\mathbf{X}}^{i})^{T}\mathbf{Q}$,%
\begin{equation}
\mathbf{H}_{1i}\left( \mathbf{X}\right) =%
\begin{bmatrix}
\left( \bar{a}_{i}\left( \mathbf{x}\right) \mathbf{I}+\breve{a}_{i}\left( 
\mathbf{x}\right) \mathbf{xx}^{T}\right) \mathbf{Q}^{T}\mathbf{Z}_{i} & 
\mathbf{0} \\ 
\mathbf{0} & \mathbf{0}%
\end{bmatrix}
\label{H1i}
\end{equation}%
\begin{eqnarray}
\mathbf{H}_{22}\left( \mathbf{X}\right)  &=&a_{2}\left( \mathbf{x}\right) 
\overline{\mathbf{ad}}_{\mathbf{Q}}\mathbf{ad}_{\mathbf{Z}}  \label{H2i} \\
\mathbf{H}_{23}\left( \mathbf{X}\right)  &=&a_{3}\left( \mathbf{x}\right)
\left( \overline{\mathbf{ad}}_{\overline{\mathbf{Q}}_{1}}\mathbf{ad}_{%
\mathbf{Z}}+\overline{\mathbf{ad}}_{\mathbf{Q}}\mathbf{ad}_{\mathbf{X}}%
\mathbf{ad}_{\mathbf{Z}}+\overline{\mathbf{ad}}_{\mathbf{Q}}\mathbf{ad}_{%
\mathbf{Z}_{1}}\right)   \notag \\
\mathbf{H}_{24}\left( \mathbf{X}\right)  &=&a_{4}\left( \mathbf{x}\right)
\left( \overline{\mathbf{ad}}_{\overline{\mathbf{Q}}_{2}}\mathbf{ad}_{%
\mathbf{Z}}+\overline{\mathbf{ad}}_{\overline{\mathbf{Q}}_{1}}\left( \mathbf{%
ad}_{\mathbf{Z}_{1}}+\mathbf{ad}_{\mathbf{X}}\mathbf{ad}_{\mathbf{Z}}\right)
\right.   \notag \\
&&\ \ \ \ \ \ \ \ \ +\left. \overline{\mathbf{ad}}_{\mathbf{Q}}\left( 
\mathbf{ad}_{\mathbf{X}}^{2}\mathbf{ad}_{\mathbf{Z}}+\mathbf{ad}_{\mathbf{X}}%
\mathbf{ad}_{\mathbf{Z}_{1}}+\mathbf{ad}_{\mathbf{Z}_{2}}\right) \right)  
\notag
\end{eqnarray}%
\begin{eqnarray}
\mathbf{H}_{31}\left( \mathbf{X}\right)  &=&-\bar{a}_{1}\left( \mathbf{x}%
\right) 
\begin{bmatrix}
\mathbf{x} \\ 
\mathbf{0}%
\end{bmatrix}%
\mathbf{Q}^{T}\mathbf{ad}_{\mathbf{Z}}  \label{H3i} \\
\mathbf{H}_{32}\left( \mathbf{X}\right)  &=&\bar{a}_{2}\left( \mathbf{x}%
\right) 
\begin{bmatrix}
\mathbf{x} \\ 
\mathbf{0}%
\end{bmatrix}%
\mathbf{Q}^{T}\left( \mathbf{ad}_{\mathbf{Z}}\mathbf{ad}_{\mathbf{X}}-2%
\mathbf{ad}_{\mathbf{X}}\mathbf{ad}_{\mathbf{Z}}\right)   \notag \\
\mathbf{H}_{33}\left( \mathbf{X}\right)  &=&-\bar{a}_{3}\left( \mathbf{x}%
\right) 
\begin{bmatrix}
\mathbf{x} \\ 
\mathbf{0}%
\end{bmatrix}%
\mathbf{Q}^{T}\left( \mathbf{ad}_{\mathbf{X}}^{2}\mathbf{ad}_{\mathbf{Z}}+%
\mathbf{ad}_{\mathbf{X}}\mathbf{ad}_{\mathbf{Z}_{1}}+\mathbf{ad}_{\mathbf{Z}%
_{2}}\right)   \notag \\
\mathbf{H}_{34}\left( \mathbf{X}\right)  &=&-\bar{a}_{4}\left( \mathbf{x}%
\right) 
\begin{bmatrix}
\mathbf{x} \\ 
\mathbf{0}%
\end{bmatrix}%
\mathbf{Q}^{T}\left( \mathbf{ad}_{\mathbf{X}}^{3}\mathbf{ad}_{\mathbf{Z}}+%
\mathbf{ad}_{\mathbf{X}}\mathbf{ad}_{\mathbf{Z}_{2}}+\mathbf{ad}_{\mathbf{X}%
}^{2}\mathbf{ad}_{\mathbf{Z}_{1}}+\mathbf{ad}_{\mathbf{Z}_{3}}\right) . 
\notag
\end{eqnarray}%
The Hessian of the function $\bar{h}^{\mathrm{inv}}\left( \mathbf{X}\right) =%
\mathbf{Q}^{T}\mathbf{dexp}_{\mathbf{X}}^{-1}\mathbf{Z}$ is $\mathbf{H}^{%
\mathrm{inv}}=\mathbf{H}_{1}^{\mathrm{inv}}+\mathbf{H}_{2}^{\mathrm{inv}}+%
\mathbf{H}_{3}^{\mathrm{inv}}$, with%
\begin{equation}
\mathbf{H}_{1}^{\mathrm{inv}}=\sum_{i=2,4}\mathbf{H}_{1i}^{\mathrm{inv}},\ \ 
\mathbf{H}_{2}^{\mathrm{inv}}=\sum_{i=2,4}\left( \mathbf{H}_{2i}^{\mathrm{inv%
}}+(\mathbf{H}_{2i}^{\mathrm{inv}})^{T}\right) ,\ \ \mathbf{H}_{3}^{\mathrm{%
inv}}=\sum_{i=2,4}\left( \mathbf{H}_{3i}^{\mathrm{inv}}+(\mathbf{H}_{3i}^{%
\mathrm{inv}})^{T}\right) 
\end{equation}%
where $\mathbf{H}_{1i}^{\mathrm{inv}},\mathbf{H}_{2i}^{\mathrm{inv}},\mathbf{%
H}_{3i}^{\mathrm{inv}}$ are the $\mathbf{H}_{1i},\mathbf{H}_{2i},\mathbf{H}%
_{3i}$ in (\ref{H1i}),(\ref{H2i}),(\ref{H3i}) with $a_{i},\bar{a}_{i},\breve{%
a}_{i}$ replaced by $b_{i},\bar{b}_{i},\breve{b}_{i}$, respectively.
\end{lemma}

\begin{proof}
With (\ref{D2ai}), the $i$th summand of the first term in (\ref{d2exp}) is
written with matrix (\ref{H1i}) as $(\mathrm{D}_{\mathbf{X}}^{2}a_{i})\left( 
\mathbf{U}\right) \left( \mathbf{S}\right) \mathbf{Q}^{T}\mathbf{P}%
_{i}\left( \mathbf{X}\right) \mathbf{Z}=\mathbf{S}^{T}\mathbf{H}_{1i}\left( 
\mathbf{X}\right) \mathbf{U}$. The $i$th summand in the second term is
written as $\mathbf{Q}^{T}(\mathrm{D}_{\mathbf{X}}^{2}\mathbf{P}_{i})\left( 
\mathbf{U}\right) \left( \mathbf{S}\right) \mathbf{Z}=\mathbf{S}^{T}\left( 
\mathbf{H}_{2i}\left( \mathbf{X}\right) +\mathbf{H}_{2i}^{T}\left( \mathbf{X}%
\right) \right) \mathbf{U}$. Explicit expressions are obtained from (\ref%
{DPi}), and some algebraic manipulation. For $i=3$, this yields $\mathbf{Q}%
^{T}(\mathrm{D}_{\mathbf{X}}^{2}\mathbf{P}_{3})\left( \mathbf{U}\right)
\left( \mathbf{S}\right) \mathbf{Z}=\mathbf{Q}^{T}(\mathbf{ad}_{\mathbf{X}}%
\mathbf{ad}_{\mathbf{S}}\mathbf{ad}_{\mathbf{U}}+\mathbf{\mathbf{ad}_{%
\mathbf{S}}\mathbf{ad}_{\mathbf{X}}ad}_{\mathbf{U}}+\mathbf{ad}_{\mathbf{U}}%
\mathbf{\mathbf{ad}_{\mathbf{S}}\mathbf{ad}_{\mathbf{X}}}+\mathbf{ad}_{%
\mathbf{U}}\mathbf{\mathbf{ad}_{\mathbf{X}}\mathbf{ad}_{\mathbf{S}}}+\mathbf{%
\mathbf{ad}_{\mathbf{S}}ad}_{\mathbf{U}}\mathbf{\mathbf{ad}_{\mathbf{X}}}+%
\mathbf{\mathbf{ad}_{\mathbf{X}}ad}_{\mathbf{U}}\mathbf{\mathbf{ad}_{\mathbf{%
S}}})\mathbf{Z}=\mathbf{S}^{T}(\overline{\mathbf{ad}}_{\overline{\mathbf{Q}}%
_{1}}\mathbf{ad}_{\mathbf{Z}}+\overline{\mathbf{ad}}_{\mathbf{Q}}\mathbf{ad}%
_{\mathbf{X}}\mathbf{ad}_{\mathbf{Z}}+\overline{\mathbf{ad}}_{\mathbf{Q}}%
\mathbf{ad}_{\mathbf{Z}_{1}}+\mathbf{ad}_{\mathbf{Z}_{1}}^{T}\overline{%
\mathbf{ad}}_{\mathbf{Q}}^{T}+\mathbf{ad}_{\mathbf{Z}}^{T}\mathbf{ad}_{%
\mathbf{X}}^{T}\overline{\mathbf{ad}}_{\mathbf{Q}}^{T}+\mathbf{ad}_{\mathbf{Z%
}}^{T}\overline{\mathbf{ad}}_{\overline{\mathbf{Q}}_{1}}^{T})\mathbf{U}$.
Expressions for $i=2,4$ are obtained similarly. The last two terms in (\ref%
{d2exp}) are written $\left( \mathrm{D}_{\mathbf{X}}a_{i}\right) \left( 
\mathbf{U}\right) \left( \mathrm{D}_{\mathbf{X}}\mathbf{P}_{i}\right) \left( 
\mathbf{S}\right) +\left( \mathrm{D}_{\mathbf{X}}a_{i}\right) \left( \mathbf{%
S}\right) \left( \mathrm{D}_{\mathbf{X}}\mathbf{P}_{i}\right) \left( \mathbf{%
U}\right) =\mathbf{S}^{T}\left( \mathbf{H}_{3i}\left( \mathbf{X}\right) +%
\mathbf{H}_{3i}^{T}\left( \mathbf{X}\right) \right) \mathbf{U}$. The
relations (\ref{H3i}) are found by application of (\ref{D2ai}) and (\ref{DPi}%
).
\end{proof}

\section{Approximation of the Derivative of Tangent Operator, and of the
Jacobian of its Evaluation Map}

\label{secApprox1}

\subsection{Series Expansions}

The directional derivative of the \textbf{dexp} matrix admits the series
expansion, deduced from (\ref{dexpMat}),%
\begin{eqnarray}
\left( \mathrm{D}_{\mathbf{X}}\mathbf{dexp}\right) (\mathbf{U})
&=&\sum_{i=0}^{\infty }\frac{1}{\left( i+1\right) !}\left( \mathrm{D}_{%
\mathbf{X}}\mathbf{P}_{i}\right) 
\hspace{-0.5ex}%
\left( \mathbf{U}\right)  \label{Ddexp} \\
\left( \mathrm{D}_{\mathbf{X}}\mathbf{dexp}^{-1}\right) (\mathbf{U})
&=&\sum_{i=0}^{\infty }\frac{B_{i}}{i!}\left( \mathrm{D}_{\mathbf{X}}\mathbf{%
P}_{i}\right) 
\hspace{-0.5ex}%
\left( \mathbf{U}\right)  \label{DdexpInv}
\end{eqnarray}%
which is completely determined with (\ref{DPi}). More involved is the series
expansion of the Jacobian of the evaluation map. Denote again $\mathbf{Z}%
_{i}:=\mathbf{P}_{i}(\mathbf{X})\mathbf{Z}$ in the following.

\begin{lemma}
\label{lemDerVec}The Jacobian of the evaluation map $f\left( \mathbf{X}%
\right) =\mathbf{dexp}_{\mathbf{X}}\mathbf{Z}$ and $f^{\mathrm{inv}}\left( 
\mathbf{X}\right) =\mathbf{dexp}_{\mathbf{X}}^{-1}\mathbf{Z}$, parameterized
by $\mathbf{Z}$, possesses the series expansions%
\begin{equation}
\frac{\partial f}{\partial \mathbf{X}}\mathbf{=}\sum_{i=0}^{\infty }\frac{1}{%
\left( i+1\right) !}\sum_{j=0}^{i-1}\mathbf{P}_{i,j}(\mathbf{X},\mathbf{Z)}%
,\ \ \frac{\partial f^{\mathrm{inv}}}{\partial \mathbf{X}}\mathbf{=}%
\sum_{i=0}^{\infty }\frac{B_{i}}{i!}\sum_{j=0}^{i-1}\mathbf{P}_{i,j}(\mathbf{%
X},\mathbf{Z)}
\end{equation}%
where%
\begin{equation}
\mathbf{P}_{i,j}(\mathbf{X},\mathbf{Z}):=-\mathbf{P}_{j}(\mathbf{X})\mathbf{%
ad}_{\mathbf{Z}_{i-j-1}}  \label{Pij}
\end{equation}%
which satisfies the recursive relation%
\begin{equation}
\mathbf{P}_{i,j}(\mathbf{X},\mathbf{Z})=\mathbf{ad}_{\mathbf{X}}\mathbf{P}%
_{i-1,j}(\mathbf{X},\mathbf{Z})+\mathbf{P}_{j}(\mathbf{X})\mathbf{ad}_{%
\mathbf{Z}_{i-j-2}}\mathbf{ad}_{\mathbf{X}}.  \label{PijRec}
\end{equation}
\end{lemma}

\begin{proof}
The $j$th term in the expression (\ref{DPi}) for $\left( \mathrm{D}_{\mathbf{%
X}}\mathbf{P}_{i}\right) 
\hspace{-0.5ex}%
\left( \mathbf{U}\right) \mathbf{Z}$ is $\mathbf{ad}_{\mathbf{X}}^{j}\mathbf{%
ad}_{\mathbf{U}}\mathbf{ad}_{\mathbf{X}}^{i-j-1}\mathbf{Z}=-\mathbf{ad}_{%
\mathbf{X}}^{j}\mathbf{ad}_{\mathbf{Z}_{i-j-1}}\mathbf{U=P}_{i,j}(\mathbf{X},%
\mathbf{Z})\mathbf{U}$, using the definition (\ref{Pij}). Writing $\mathbf{P}%
_{i,j}(\mathbf{X},\mathbf{Z})=\mathbf{ad}_{\mathbf{X}}^{j}\mathbf{ad}_{%
\mathbf{U}}\mathbf{ad}_{\mathbf{X}}\mathbf{ad}_{\mathbf{X}}^{i-j-2}\mathbf{Z}
$, and application of the Jacobi identity to the triple Lie bracket on the
right yields $\mathbf{P}_{i,j}(\mathbf{X},\mathbf{Z})=-\mathbf{ad}_{\mathbf{X%
}}^{j}%
\big%
(\mathbf{ad}_{\mathbf{X}}\mathbf{ad}_{\mathbf{ad}_{\mathbf{X}}^{i-j-2}%
\mathbf{Z}}\mathbf{U}+\mathbf{ad}_{\mathbf{ad}_{\mathbf{X}}^{i-j-2}\mathbf{Z}%
}\mathbf{ad}_{\mathbf{U}}\mathbf{X}%
\big%
)=%
\big%
(-\mathbf{ad}_{\mathbf{X}}^{j+1}\mathbf{ad}_{\mathbf{ad}_{\mathbf{X}}^{i-j-2}%
\mathbf{Z}}\mathbf{U}+\mathbf{ad}_{\mathbf{X}}^{j}\mathbf{ad}_{\mathbf{ad}_{%
\mathbf{X}}^{i-j-2}\mathbf{Z}}\mathbf{ad}_{\mathbf{X}}%
\big%
)\mathbf{U}$, and thus (\ref{PijRec}).
\end{proof}

\subsection{Higher-Order Approximations}

Often approximations of the dexp map and its derivatives are sufficient, or
even necessary, e.g. when switching must be avoided to ensure continuity
within optimization schemes such as differential dynamic programing (DDP).
In the following, the relations up to third order are derived. To this end,
the necessary directional derivatives of $\mathbf{P}_{i}=\mathbf{ad}_{%
\mathbf{X}}^{j}$ and of the sum of $\mathbf{P}_{i,j}$ for fixed $i$ are
listed in table \ref{tabDP}. The derivative of $\mathbf{P}_{i}$ is of order $%
i-1$ in $\mathbf{X}$, and a $k$th-order approximation is obtained by
truncating the series expansions (\ref{Ddexp}) at $i=k+1$. Approximations of
order $k=0,\ldots ,3$ of the derivative of the $\mathbf{dexp}$ matrix in (%
\ref{dexpMat}) and of the Jacobian of the evaluation map $\mathbf{dexp}_{%
\mathbf{X}}\mathbf{Z}$ are listed in table \ref{tabDdexp}. Table \ref%
{tabDdexpInv} shows approximations when the inverse of $\mathbf{dexp}$ in (%
\ref{dexpMat}) is used. 
\color[rgb]{0,0,0}%
Notice that the 1st- and 2nd-order approximations of the derivative of $%
\mathbf{dexp}^{-1}$ and $\mathbf{dexp}_{\mathbf{X}}^{-1}\mathbf{Z}$ are
identical since the odd Bernoulli numbers $B_{i}$ are zero for $i>1$.%
\color{black}%

\begin{table}[ht] \centering%
\begin{tabular}{|l|l|l|}
\hline
$i$ & $\left( \mathrm{D}_{\mathbf{X}}\mathbf{P}_{i}\right) 
\hspace{-0.5ex}%
\left( \mathbf{U}\right) $ & $\sum_{j=0}^{i-1}\mathbf{P}_{i,j}(\mathbf{X},%
\mathbf{Z)}$ \\ \hline
1 & $\mathbf{ad}_{\mathbf{U}}$ & $-\mathbf{ad}_{\mathbf{Z}}$ \\ 
2 & $\mathbf{ad}_{\mathbf{U}}\mathbf{P}_{1}+\mathbf{P}_{1}\mathbf{ad}_{%
\mathbf{U}}$ & $-\mathbf{ad}_{\mathbf{Z}_{1}}-\mathbf{P}_{1}\mathbf{ad}_{%
\mathbf{Z}}$ \\ 
3 & $\mathbf{ad}_{\mathbf{U}}\mathbf{P}_{2}+\mathbf{P}_{1}\mathbf{ad}_{%
\mathbf{U}}\mathbf{P}_{1}+\mathbf{P}_{2}\mathbf{ad}_{\mathbf{U}}$ & $-%
\mathbf{ad}_{\mathbf{Z}_{2}}-\mathbf{P}_{1}\mathbf{ad}_{\mathbf{Z}_{1}}-%
\mathbf{P}_{2}\mathbf{ad}_{\mathbf{Z}}$ \\ 
4 & $\mathbf{ad}_{\mathbf{U}}\mathbf{P}_{3}+\mathbf{P}_{2}\mathbf{ad}_{%
\mathbf{U}}\mathbf{P}_{1}+\mathbf{P}_{1}\mathbf{ad}_{\mathbf{U}}\mathbf{P}%
_{2}$ & $-\mathbf{ad}_{\mathbf{Z}_{3}}-\mathbf{P}_{1}\mathbf{ad}_{\mathbf{Z}%
_{2}}-\mathbf{P}_{2}\mathbf{ad}_{\mathbf{Z}_{1}}-\mathbf{P}_{3}\mathbf{ad}_{%
\mathbf{Z}}$ \\ 
& $\ \ \ \ \ \ \ \ \ \ \ +\mathbf{P}_{3}\mathbf{\mathbf{a}d}_{\mathbf{U}}$ & 
\\ \hline
\end{tabular}%
\caption{Directional derivatives of ${\bf P}_{i}({\bf X})={\bf ad}_{X}^{i}$,
and the sum of ${\bf P}_{i,j}$ in (\ref{PijRec}) for fixed $i=1,\ldots,4$.}%
\vspace{-2ex}\label{tabDP}%
\end{table}%

\begin{table}[ht] \centering%
\begin{tabular}{|l|l|l|}
\cline{1-2}\cline{3-3}
$k$ & $(\mathrm{D}_{\mathbf{X}}^{\left[ k\right] }\mathbf{dexp})(\mathbf{U})$
& $\frac{\partial ^{\left[ k\right] }}{\partial \mathbf{X}}\mathbf{dexp}_{%
\mathbf{X}}\mathbf{Z}$ \\ \cline{1-2}\cline{3-3}
0 & $\frac{1}{2}\mathbf{ad}_{\mathbf{U}}$ & $-\frac{1}{2}\mathbf{ad}_{%
\mathbf{Z}}$ \\ 
1 & $\frac{1}{2}\mathbf{ad}_{\mathbf{U}}+\frac{1}{6}(\mathbf{ad}_{\mathbf{U}}%
\mathbf{P}_{1}+\mathbf{P}_{1}\mathbf{ad}_{\mathbf{U}})$ & $-\frac{1}{2}%
\mathbf{ad}_{\mathbf{Z}}-\frac{1}{6}(\mathbf{ad}_{\mathbf{Z}_{1}}+\mathbf{P}%
_{1}\mathbf{ad}_{\mathbf{Z}})$ \\ 
2 & $(\mathrm{D}_{\mathbf{X}}^{\left[ 2\right] }\mathbf{dexp})(\mathbf{U})+%
\frac{1}{24}(\mathbf{ad}_{\mathbf{U}}\mathbf{P}_{2}+\mathbf{P}_{1}\mathbf{ad}%
_{\mathbf{U}}\mathbf{P}_{1}+\mathbf{P}_{2}\mathbf{ad}_{\mathbf{U}})$ & $%
\frac{\partial ^{\left[ 1\right] }}{\partial \mathbf{X}}\mathbf{dexp}_{%
\mathbf{X}}\mathbf{Z}-\frac{1}{24}(\mathbf{ad}_{\mathbf{Z}_{2}}+\mathbf{P}%
_{1}\mathbf{ad}_{\mathbf{Z}_{1}}+\mathbf{P}_{2}\mathbf{ad}_{\mathbf{Z}})$ \\ 
3 & $(\mathrm{D}_{\mathbf{X}}^{\left[ 3\right] }\mathbf{dexp})(\mathbf{U})+%
\frac{1}{120}\left( \mathbf{ad}_{\mathbf{U}}\mathbf{P}_{3}+\mathbf{P}_{2}%
\mathbf{ad}_{\mathbf{U}}\mathbf{P}_{1}+\mathbf{P}_{1}\mathbf{ad}_{\mathbf{U}}%
\mathbf{P}_{2}+\mathbf{P}_{3}\mathbf{ad}_{\mathbf{U}}\right) $ & $\frac{%
\partial ^{\left[ 2\right] }}{\partial \mathbf{X}}\mathbf{dexp}_{\mathbf{X}}%
\mathbf{Z}$ \\ \cline{1-2}\cline{3-3}
\end{tabular}%
\caption{Approximations of order $k=0,\ldots,3$ of the directional derivate of $\bf 
dexp$ and of the Jacobian of its evaluation map}\vspace{-2ex}\label{tabDdexp}%
\end{table}%

\begin{table}[h!] \centering%
\begin{tabular}{|l|l|l|}
\hline
$k$ & $(\mathrm{D}_{\mathbf{X}}^{\left[ k\right] }\mathbf{dexp}^{-1})(%
\mathbf{U})$ & $\frac{\partial ^{\left[ k\right] }}{\partial \mathbf{X}}%
\mathbf{dexp}_{\mathbf{X}}^{-1}\mathbf{Z}$ \\ \hline
0 & $-\frac{1}{2}\mathbf{ad}_{\mathbf{U}}$ & $\frac{1}{2}\mathbf{ad}_{%
\mathbf{Z}}$ \\ 
1 & $-\frac{1}{2}\mathbf{ad}_{\mathbf{U}}+\frac{1}{12}(\mathbf{ad}_{\mathbf{U%
}}\mathbf{P}_{1}+\mathbf{P}_{1}\mathbf{ad}_{\mathbf{U}})$ & $\frac{1}{2}%
\mathbf{ad}_{\mathbf{Z}}-\frac{1}{12}(\mathbf{ad}_{\mathbf{Z}_{1}}+\mathbf{P}%
_{1}\mathbf{ad}_{\mathbf{Z}})$ \\ 
3 & $(\mathrm{D}_{\mathbf{X}}^{\left[ 1\right] }\mathbf{dexp}^{-1})(\mathbf{U%
})-\frac{1}{720}\left( \mathbf{ad}_{\mathbf{U}}\mathbf{P}_{3}+\mathbf{P}_{2}%
\mathbf{ad}_{\mathbf{U}}\mathbf{P}_{1}+\mathbf{P}_{1}\mathbf{ad}_{\mathbf{U}}%
\mathbf{P}_{2}+\mathbf{P}_{3}\mathbf{ad}_{\mathbf{U}}\right) $ & $\frac{%
\partial ^{\left[ 1\right] }}{\partial \mathbf{X}}\mathbf{dexp}_{\mathbf{X}%
}^{-1}\mathbf{Z+}\frac{1}{720}\left( \mathbf{ad}_{\mathbf{Z}_{3}}+\mathbf{P}%
_{1}\mathbf{ad}_{\mathbf{Z}_{2}}+\mathbf{P}_{2}\mathbf{ad}_{\mathbf{Z}_{1}}+%
\mathbf{P}_{3}\mathbf{ad}_{\mathbf{Z}}\right) $ \\ \hline
\end{tabular}%
\caption{Approximations of order $k=0,1,3$ of the directional derivative
of ${\bf dexp}^{-1}$ of the Jacobian of its. \color[rgb]{0,0,0}Notice that the 1st-order is identical to 2nd-order approximation (not shown).}%
\label{tabDdexpInv}%
\end{table}%

\section{Approximation of the Second Derivative of Tangent Operator, and of
the Hessian of its Evaluation Map}

\label{secApprox2}

\subsection{Series expansions}

The second directional derivatives of the series expansion (\ref{Ddexp}) and
(\ref{DdexpInv}) of $\mathbf{dexp}$ and $\mathbf{dexp}^{-1}$, respectively,
along $\mathbf{U},\mathbf{S}\in {\mathbb{R}}^{n}$ are%
\begin{eqnarray}
\left( \mathrm{D}_{\mathbf{X}}^{2}\mathbf{dexp}\right) (\mathbf{U})(\mathbf{S%
}) &=&\sum_{i=2}^{\infty }\frac{1}{\left( i+1\right) !}\left( \mathrm{D}_{%
\mathbf{X}}^{2}\mathbf{P}_{i}\right) 
\hspace{-0.5ex}%
\left( \mathbf{U}\right) (\mathbf{S})  \label{DDdexp} \\
\left( \mathrm{D}_{\mathbf{X}}^{2}\mathbf{dexp}^{-1}\right) (\mathbf{U})(%
\mathbf{S}) &=&\sum_{i=2}^{\infty }\frac{B_{i}}{i!}\left( \mathrm{D}_{%
\mathbf{X}}^{2}\mathbf{P}_{i}\right) 
\hspace{-0.5ex}%
\left( \mathbf{U}\right) (\mathbf{S}).  \label{DDdexpInv}
\end{eqnarray}%
They again boil down to the second derivatives of $\mathbf{P}_{i}$ that are
determined with Lemma \ref{lemDerPi}, and given in (\ref{dPi}) for $%
i=1,\ldots ,4$. Postmultiplication of (\ref{DDdexp}) and (\ref{DDdexpInv})
with $\mathbf{Z}$ yields the second derivative of the vector functions $%
f\left( \mathbf{X}\right) =\mathbf{dexp}_{\mathbf{X}}\mathbf{Z}$ and $f^{%
\mathrm{inv}}\left( \mathbf{X}\right) =\mathbf{dexp}_{\mathbf{X}}^{-1}%
\mathbf{Z}$. In various applications, the Hessian of the scalar functions
defined as $h\left( \mathbf{X}\right) =\mathbf{Q}^{T}\mathbf{dexp}_{\mathbf{X%
}}\mathbf{Z}$ and $h^{\mathrm{inv}}\left( \mathbf{X}\right) =\mathbf{Q}^{T}%
\mathbf{dexp}_{\mathbf{X}}^{-1}\mathbf{Z}$, with some vector $\mathbf{Z},%
\mathbf{Q}\in {\mathbb{R}}^{n}$, are needed. The respective Hessian is the
matrix $\mathbf{H}=\left[ \frac{\partial ^{2}h}{\partial X_{i}\partial X_{j}}%
\right] $ and $\mathbf{H}^{\mathrm{inv}}=\left[ \frac{\partial ^{2}h^{%
\mathrm{inv}}}{\partial X_{i}\partial X_{j}}\right] $, such that $\left( 
\mathrm{D}_{\mathbf{X}}^{2}h\right) (\mathbf{U})(\mathbf{S})=\mathbf{U}^{T}%
\mathbf{HS}$ and $\left( \mathrm{D}_{\mathbf{X}}^{2}h\right) (\mathbf{U})(%
\mathbf{S})=\mathbf{U}^{T}\mathbf{H}^{\mathrm{inv}}\mathbf{S}$, respectively.

\begin{lemma}
\label{lemDer2}The second directional derivatives of $\mathbf{P}_{i}$ can be
expressed as%
\begin{equation}
\left( \mathrm{D}_{\mathbf{X}}^{2}\mathbf{P}_{i}\right) 
\hspace{-0.5ex}%
\left( \mathbf{U}\right) (\mathbf{S})\mathbf{Z}=\sum_{j=1}^{i-1}%
\sum_{l=0}^{j-1}\left( \mathbf{P}_{j,l}(\mathbf{X},\mathbf{ad}_{\mathbf{U}}%
\mathbf{P}_{i-j-1}(\mathbf{X})\mathbf{Z)S}\left. +\mathbf{P}_{i-j-1}(\mathbf{%
X})\mathbf{ad}_{\mathbf{U}}\mathbf{P}_{j,l}(\mathbf{X},\mathbf{Z})\mathbf{S}%
\right) \right) ,i\geq 2.  \label{D2P}
\end{equation}%
The Hessian $\mathbf{H}=\left[ \frac{\partial ^{2}h}{\partial X_{i}\partial
X_{j}}\right] $ and $\mathbf{H}^{\mathrm{inv}}=\left[ \frac{\partial ^{2}h^{%
\mathrm{inv}}}{\partial X_{i}\partial X_{j}}\right] $ of function $h\left( 
\mathbf{X}\right) :=\mathbf{Q}^{T}\mathbf{dexp_{\mathbf{X}}Z}$ and $h^{%
\mathrm{inv}}\left( \mathbf{X}\right) :=\mathbf{Q}^{T}\mathbf{dexp}_{\mathbf{%
X}}^{-1}\mathbf{Z}$, respectively, parameterized by $\mathbf{Z}$ and $%
\mathbf{Q}$, are%
\begin{equation}
\mathbf{H}\left( \mathbf{X}\right) =\sum_{i=2}^{\infty }\frac{1}{\left(
i+1\right) !}\mathbf{H}_{i}\left( \mathbf{X}\right) ,\ \ \mathbf{H}^{\mathrm{%
inv}}\left( \mathbf{X}\right) =\sum_{i=2}^{\infty }\frac{B_{i}}{i!}\mathbf{H}%
_{i}\left( \mathbf{X}\right)  \label{H}
\end{equation}%
where the Hessian $\mathbf{H}_{i}(\mathbf{X})$ of $\mathbf{Q}^{T}\mathbf{P}%
_{i}\left( \mathbf{X}\right) \mathbf{Z}$ is decomposed as $\mathbf{H}_{i}(%
\mathbf{X})=\overline{\mathbf{H}}_{i}(\mathbf{X})+\overline{\mathbf{H}}%
_{i}^{T}(\mathbf{X})$, and $\overline{\mathbf{H}}_{i}$ are defined with $%
\overline{\mathbf{ad}}_{\mathbf{Q}}$ in (\ref{adBar}) as%
\begin{equation}
\overline{\mathbf{H}}_{i}(\mathbf{X})=\sum_{j=1}^{i-1}\sum_{l=0}^{j-1}%
\overline{\mathbf{ad}}_{\overline{\mathbf{Q}}_{i-j-1}}\mathbf{P}_{l}\mathbf{%
ad}_{\mathbf{Z}_{j-l-1}}=-\sum_{j=1}^{i-1}\sum_{l=0}^{j-1}\overline{\mathbf{%
ad}}_{\overline{\mathbf{Q}}_{i-j-1}}\mathbf{P}_{j,l}(\mathbf{X,Z}),i\geq 2.
\label{Hess2}
\end{equation}
\end{lemma}

\begin{proof}
The derivative of (\ref{P}) can be written as%
\begin{equation*}
\left( \mathrm{D}_{\mathbf{X}}^{2}\mathbf{P}_{i}\right) \left( \mathbf{U}%
\right) \left( \mathbf{S}\right) =\sum_{j=1}^{i-1}\left( \left( \mathrm{D}_{%
\mathbf{X}}\mathbf{P}_{j}\right) \left( \mathbf{S}\right) \mathbf{ad}_{%
\mathbf{U}}\mathbf{P}_{i-j-1}\left( \mathbf{X}\right) +\mathbf{P}%
_{i-j-1}\left( \mathbf{X}\right) \mathbf{ad}_{\mathbf{U}}\left( \mathrm{D}_{%
\mathbf{X}}\mathbf{P}_{j}\right) \left( \mathbf{S}\right) \right) .
\end{equation*}%
Evaluating this with $\mathbf{P}_{i}\mathbf{Z}$ and (\ref{Pij}) leads to (%
\ref{D2P}). Now consider the second term in (\ref{D2P}). Premultiplication
with $\mathbf{Q}^{T}$ gives%
\begin{align*}
& \mathbf{Q}^{T}\mathbf{P}_{i-j-1}\left( \mathbf{X}\right) \mathbf{ad}_{%
\mathbf{U}}\mathbf{P}_{j,l}\left( \mathbf{X},\mathbf{Z}\right) \mathbf{S}=-%
\mathbf{Q}^{T}\mathbf{P}_{i-j-1}\left( \mathbf{X}\right) \mathbf{ad}_{%
\mathbf{U}}\mathbf{ad}_{\mathbf{X}}^{l}\mathbf{ad}_{\mathbf{ad}_{\mathbf{X}%
}^{j-l-1}\mathbf{Z}}\mathbf{S} \\
& =-\left( \mathbf{ad}_{\mathbf{U}}^{T}\mathbf{P}_{i-j-1}^{T}\left( \mathbf{X%
}\right) \mathbf{Q}\right) ^{T}\mathbf{ad}_{\mathbf{X}}^{l}\mathbf{ad}_{%
\mathbf{ad}_{\mathbf{X}}^{j-l-1}\mathbf{Z}}\mathbf{S}=\mathbf{U}^{T}%
\overline{\mathbf{ad}}_{\mathbf{P}_{i-j-1}^{T}\left( \mathbf{X}\right) 
\mathbf{Q}}\mathbf{ad}_{\mathbf{X}}^{l}\mathbf{ad}_{\mathbf{ad}_{\mathbf{X}%
}^{j-l-1}\mathbf{Z}}\mathbf{S} \\
& =\mathbf{U}^{T}\overline{\mathbf{ad}}_{\overline{\mathbf{Q}}_{i-j-1}}%
\mathbf{P}_{l}\mathbf{ad}_{\mathbf{Z}_{j-l-1}}\mathbf{S}=\mathbf{Q}^{T}%
\overline{\mathbf{H}}_{i}(\mathbf{X})\mathbf{S}
\end{align*}%
with $\overline{\mathbf{H}}_{i}(\mathbf{X})$ in (\ref{Hess2}). A similar
derivation of the first term in (\ref{D2P}) yields $\overline{\mathbf{H}}%
_{i}^{T}(\mathbf{X})$.
\end{proof}

\subsection{Higher-Order Approximations}

Approximations $(\mathrm{D}_{\mathbf{X}}^{2\left[ k\right] }\mathbf{dexp})(%
\mathbf{U})(\mathbf{S})$ of order $k=0,1,2$ of the second directional
derivative $(\mathrm{D}_{\mathbf{X}}^{2}\mathbf{dexp})(\mathbf{U})(\mathbf{S}%
)$ are found from (\ref{DDdexp}) along with the explicit relations for $%
\left( \mathrm{D}_{\mathbf{X}}^{2}\mathbf{P}_{i}\right) $ in (\ref{d2Pi}),
and from (\ref{DDdexpInv}) for the inverse. In view of (\ref{d2Pi}), the $k$%
th-order approximation involves the terms with $\left( \mathrm{D}_{\mathbf{X}%
}^{2}\mathbf{P}_{i}\right) $ 
\color[rgb]{0,0,0}%
for $i=0,\ldots ,k+2$%
\color{black}%
. Because the odd Bernoulli numbers $B_{i}$ are zero, there are no
approximations of $(\mathrm{D}_{\mathbf{X}}^{2}\mathbf{dexp}^{-1})(\mathbf{U}%
)(\mathbf{S})$ of order $k=1,3,5,\ldots $. Approximations up to 2nd-order
are summarized in table \ref{tabD2dexp} and \ref{tabD2dexpInv},
respectively. The second derivative of the evaluation map is obtained by
postmultiplying $(\mathrm{D}_{\mathbf{X}}^{2\left[ k\right] }\mathbf{dexp})(%
\mathbf{U})(\mathbf{S})$ with $\mathbf{Z}$.

Low-order approximations of the Hessian are constructed with the matrices $%
\overline{\mathbf{H}}_{i}$ according to (\ref{Hess2}). The necessary $%
\sum_{l=0}^{j-1}\mathbf{P}_{i-j-1,l}$ are available from table \ref{tabDP}.
Matrices $\overline{\mathbf{H}}_{i},i=0,1,2$ are listed in table \ref%
{tabHessian}. The corresponding approximations of the Hessian of $h$ and $h^{%
\mathrm{inv}}$ are constructed with (\ref{H}). With these relations, the
approximations $\mathbf{H}^{\left[ k\right] }\left( \mathbf{X}\right) $ of
order $k=0,\ldots ,2$ are (omitting argument $\mathbf{X}$)%
\begin{eqnarray}
\mathbf{H}^{\left[ 0\right] } &=&\frac{1}{6}(\overline{\mathbf{H}}_{2}+%
\overline{\mathbf{H}}_{2}^{T}),\ \mathbf{H}^{\left[ 1\right] }=\mathbf{H}^{%
\left[ 0\right] }+\frac{1}{24}(\overline{\mathbf{H}}_{3}+\overline{\mathbf{H}%
}_{3}^{T}),\ \mathbf{H}^{\left[ 3\right] }=\mathbf{H}^{\left[ 2\right] }+%
\frac{1}{120}(\overline{\mathbf{H}}_{4}+\overline{\mathbf{H}}_{4}^{T}) \\
\mathbf{H}^{\mathrm{inv}\left[ 0\right] } &=&\frac{1}{12}(\overline{\mathbf{H%
}}_{2}+\overline{\mathbf{H}}_{2}^{T}),\ \ \mathbf{H}^{\mathrm{inv}\left[ 2%
\right] }=\mathbf{H}^{\mathrm{inv}\left[ 0\right] }-\frac{1}{720}(\overline{%
\mathbf{H}}_{4}+\overline{\mathbf{H}}_{4}^{T}).
\end{eqnarray}%
\begin{table}[ht] \centering%
\begin{tabular}{|l|l|}
\hline
$k$ & $(\mathrm{D}_{\mathbf{X}}^{2\left[ k\right] }\mathbf{dexp})(\mathbf{U}%
)(\mathbf{S})$ \\ \hline
0 & $\frac{1}{6}(\mathbf{ad}_{\mathbf{S}}\mathbf{ad}_{\mathbf{U}}+\mathbf{ad}%
_{\mathbf{U}}\mathbf{ad}_{\mathbf{S}})$ \\ 
1 & $\frac{1}{6}(\mathbf{ad}_{\mathbf{S}}\mathbf{ad}_{\mathbf{U}}+\mathbf{ad}%
_{\mathbf{U}}\mathbf{ad}_{\mathbf{S}})+\frac{1}{24}(\mathbf{P}_{1}\mathbf{ad}%
_{\mathbf{S}}+\mathbf{ad}_{\mathbf{S}}\mathbf{P}_{1})\mathbf{ad}_{\mathbf{U}%
}+\frac{1}{24}(\mathbf{ad}_{\mathbf{U}}\mathbf{ad}_{\mathbf{S}}+\mathbf{ad}_{%
\mathbf{S}}\mathbf{ad}_{\mathbf{U}})\mathbf{P}_{1})+\frac{1}{24}(\mathbf{ad}%
_{\mathbf{U}}\mathbf{P}_{1}+\mathbf{P}_{1}\mathbf{ad}_{\mathbf{U}})\mathbf{ad%
}_{\mathbf{S}}$ \\ 
2 & $(\mathrm{D}_{\mathbf{X}}^{2\left[ 1\right] }\mathbf{dexp})(\mathbf{U})(%
\mathbf{S})+\frac{1}{120}\mathbf{ad}_{\mathbf{S}}(\mathbf{P}_{2}\mathbf{ad}_{%
\mathbf{U}}+\mathbf{ad}_{\mathbf{U}}\mathbf{P}_{2}+\mathbf{P}_{1}\mathbf{ad}%
_{\mathbf{U}}\mathbf{P}_{1})\mathbf{+}\frac{1}{120}\mathbf{ad}_{\mathbf{U}%
}\left( \mathbf{P}_{2}\mathbf{ad}_{\mathbf{S}}+\mathbf{ad}_{\mathbf{S}}%
\mathbf{P}_{2}+\mathbf{P}_{1}\mathbf{ad}_{\mathbf{S}}\mathbf{P}_{1}\right) $
\\ 
& $\ \ \ \ \ \ \ \ \ \ \ \ \ \ \ \ \ \ \ \ \ \ \ \ \ \ \ \ \ \ +\frac{1}{120}%
\mathbf{P}_{2}(\mathbf{ad}_{\mathbf{S}}\mathbf{ad}_{\mathbf{U}}+\mathbf{ad}_{%
\mathbf{U}}\mathbf{ad}_{\mathbf{S}})\mathbf{+}\frac{1}{120}\mathbf{P}_{1}(%
\mathbf{ad}_{\mathbf{S}}\left( \mathbf{P}_{1}\mathbf{ad}_{\mathbf{U}}+%
\mathbf{ad}_{\mathbf{U}}\mathbf{P}_{1}\right) +\mathbf{ad}_{\mathbf{U}%
}\left( \mathbf{P}_{1}\mathbf{ad}_{\mathbf{S}}+\mathbf{ad}_{\mathbf{S}}%
\mathbf{P}_{1}\right) )$ \\ \hline
\end{tabular}%
\caption{Approximations of order $k=0,1,2$ of the second directional derivatives 
of $(\mathrm{D}_{\mathbf{X}}^{2}\mathbf{dexp})(\mathbf{U})(\mathbf{S})$.}%
\label{tabD2dexp}%
\end{table}%

\begin{table}[ht] \centering%
\begin{tabular}{|l|l|}
\hline
$k$ & $(\mathrm{D}_{\mathbf{X}}^{2\left[ k\right] }\mathbf{dexp}^{-1})(%
\mathbf{U})(\mathbf{S})$ \\ \hline
0 & $\frac{1}{12}(\mathbf{ad}_{\mathbf{S}}\mathbf{ad}_{\mathbf{U}}+\mathbf{ad%
}_{\mathbf{U}}\mathbf{ad}_{\mathbf{S}})$ \\ 
2 & $(\mathrm{D}_{\mathbf{X}}^{2\left[ 1\right] }\mathbf{dexp}^{-1})(\mathbf{%
U})(\mathbf{S})-\frac{1}{720}\mathbf{ad}_{\mathbf{S}}(\mathbf{P}_{2}\mathbf{%
ad}_{\mathbf{U}}+\mathbf{ad}_{\mathbf{U}}\mathbf{P}_{2}+\mathbf{P}_{1}%
\mathbf{ad}_{\mathbf{U}}\mathbf{P}_{1})-\frac{1}{720}\mathbf{ad}_{\mathbf{U}%
}\left( \mathbf{P}_{2}\mathbf{ad}_{\mathbf{S}}+\mathbf{ad}_{\mathbf{S}}%
\mathbf{P}_{2}+\mathbf{P}_{1}\mathbf{ad}_{\mathbf{S}}\mathbf{P}_{1}\right) $
\\ 
& $\ \ \ \ \ \ \ \ \ \ \ \ \ \ \ \ \ \ \ \ \ \ \ \ \ \ \ \ \ \ \ \ \ \ -%
\frac{1}{720}\mathbf{P}_{2}(\mathbf{ad}_{\mathbf{S}}\mathbf{ad}_{\mathbf{U}}+%
\mathbf{ad}_{\mathbf{U}}\mathbf{ad}_{\mathbf{S}})-\frac{1}{720}\mathbf{P}%
_{1}(\mathbf{ad}_{\mathbf{S}}\left( \mathbf{P}_{1}\mathbf{ad}_{\mathbf{U}}+%
\mathbf{ad}_{\mathbf{U}}\mathbf{P}_{1}\right) +\mathbf{ad}_{\mathbf{U}%
}\left( \mathbf{P}_{1}\mathbf{ad}_{\mathbf{S}}+\mathbf{ad}_{\mathbf{S}}%
\mathbf{P}_{1}\right) )$ \\ \hline
\end{tabular}%
\caption{Approximations of order $k=0,2$ of the second directional derivatives 
of $(\mathrm{D}_{\mathbf{X}}^{2}\mathbf{dexp}^{-1})(\mathbf{U})(\mathbf{S})$. Notice that there is no 1st-order approximation.}%
\label{tabD2dexpInv}%
\end{table}%

\begin{table}[ht] \centering%
\begin{tabular}{|l|l|}
\hline
$i$ & $\overline{\mathbf{H}}_{i}(\mathbf{X})$ \\ \hline
0 & $\overline{\mathbf{ad}}_{\mathbf{Q}}\mathbf{ad}_{\mathbf{Z}}$ \\ 
1 & $\overline{\mathbf{ad}}_{\mathbf{Q}}\left( \mathbf{P}_{1}\mathbf{ad}_{%
\mathbf{Z}}+\mathbf{ad}_{\mathbf{Z}_{1}}\right) +\overline{\mathbf{ad}}_{%
\overline{\mathbf{Q}}_{1}}\mathbf{ad}_{\mathbf{Z}}$ \\ 
2 & $\overline{\mathbf{ad}}_{\mathbf{Q}}\left( \mathbf{P}_{2}\mathbf{ad}_{%
\mathbf{Z}}+\mathbf{P}_{1}\mathbf{ad}_{\mathbf{Z}_{1}}+\mathbf{ad}_{\mathbf{Z%
}_{2}}\right) +\overline{\mathbf{ad}}_{\overline{\mathbf{Q}}_{1}}\left( 
\mathbf{P}_{1}\mathbf{ad}_{\mathbf{Z}}+\mathbf{ad}_{\mathbf{Z}_{1}}\right) +%
\overline{\mathbf{ad}}_{\overline{\mathbf{Q}}_{2}}\mathbf{ad}_{\mathbf{Z}}$
\\ \hline
\end{tabular}%
\caption{Explicit expressions for the matrix
$\overline{\mathbf{H}}_{i}(\mathbf{X})$ in (\ref{Hess2}) for $i=2,\ldots,4$.}%
\label{tabHessian}%
\end{table}%

\subsection{Robust Evaluation exploiting Local Approximations}

\label{secApprox}

All closed form expressions potentially exhibit numerical issues when
evaluate at or near $\left\Vert \mathbf{x}\right\Vert =\mathbf{0}$, which is
due to the non-simply connectedness of $SO(3)$, caused by division of $%
\varphi =\left\Vert \mathbf{x}\right\Vert $ and its powers. The standard
approach to tackle this issue is to switch to the respective limit value
when $\left\Vert \mathbf{x}\right\Vert \leq \varepsilon $ for a specified
threshold $\epsilon $. This tends to introduce discontinuities, however. As
an alternative, it was proposed in \cite{TodescoBruls2023} to switch to the
series expansion when the exponential is evaluated as $\exp \tilde{\mathbf{x}%
}=\mathbf{I}+\frac{\sin \varphi }{\varphi }\tilde{\mathbf{x}}+\frac{1-\cos
\varphi }{\varphi ^{2}}\tilde{\mathbf{x}}^{2}$ and its differential as $%
\mathbf{dexp}_{\mathbf{x}}=\mathbf{I}+\frac{1-\cos \varphi }{\varphi ^{2}}%
\tilde{\mathbf{x}}+\frac{1-\frac{1}{\varphi }\sin \varphi }{\varphi ^{2}}%
\tilde{\mathbf{x}}^{2}$, rather than using the robust formulations (\ref%
{expSO3}) and (\ref{dexpSO32}). A computationally robust approach, however,
consists in switching between the explicit expressions and a higher-order
local approximation. For example, the exact derivative (\ref{DdexpSE3}) is
replaced by its $k$th-order approximation in Tab. \ref{tabDdexp} when $%
\left\Vert \mathbf{x}\right\Vert \leq \epsilon $. The threshold $\epsilon $,
that remains to by specified, depends on the application.

To investigate the actual accuracy of the presented $k$th-order
approximations, numerical results for the approximations errors are shown
for increasing $\left\Vert \mathbf{x}\right\Vert $. To this end, introduce
the screw coordinate vector $\mathbf{X}\left( s\right) =s\left[ 
\begin{smallmatrix}
\mathbf{n} \\ 
\mathbf{y}%
\end{smallmatrix}%
\right] $, where $\mathbf{n}:=\mathbf{x/}\left\Vert \mathbf{x}\right\Vert $
is a unit vector defined by the arbitrary vectors $\mathbf{x}=\left[
0.3,-0.4,1\right] $ and $\mathbf{y}=\left[ 0.1,0.2,-0.4\right] $. Using the
unit vector $\mathbf{n}$ allows controlling the norm $\left\Vert \mathbf{x}%
\right\Vert $. Further, arbitrary vectors $\mathbf{U}=\left[
0.1,1.0,-1.0,0.9,0.5,0.3\right] $ and $\mathbf{S}=\left[
1.7,-2.9,-9.2,7.6,6.7,2.4\right] $ are introduced.

The errors $\varepsilon ^{\left[ k\right] }\left( s\right) :=%
\big\|%
\mathbf{dexp}_{\mathbf{X}\left( s\right) }^{\left[ k\right] }-\mathbf{dexp}_{%
\mathbf{X}\left( s\right) }%
\big\|%
$ of the $k$th-order approximation for $s=0,\ldots ,0.1$, and when
approximating the inverse, are shown in Fig. \ref{figdexpApprox}, where $%
\left\Vert \cdot \right\Vert $ is the matrix $L_{2}$ norm. The errors $%
\varepsilon _{1}^{\left[ k\right] }\left( s\right) :=%
\big\|%
(\mathrm{D}_{\mathbf{X}\left( s\right) }^{\left[ k\right] }\mathbf{dexp}%
)\left( \mathbf{U}\right) -(\mathrm{D}_{\mathbf{X}\left( s\right) }\mathbf{%
dexp})\left( \mathbf{U}\right) 
\big\|%
$ of the $k$th-order approximation of the first derivative are shown in Fig. %
\ref{figDdexpApprox013}, where $\left\Vert \cdot \right\Vert $ is the
Euclidean vector norm. The errors $\varepsilon _{2}^{\left[ k\right] }\left(
s\right) :=%
\big\|%
(\mathrm{D}_{\mathbf{X}\left( s\right) }^{2\left[ k\right] }\mathbf{dexp}%
)\left( \mathbf{U}\right) \left( \mathbf{S}\right) -(\mathrm{D}_{\mathbf{X}%
\left( s\right) }^{2}\mathbf{dexp})\left( \mathbf{U}\right) \left( \mathbf{S}%
\right) 
\big\|%
$ made by approximating the second derivative are shown in Fig. \ref%
{figD2dexpApprox012}, where again $\left\Vert \cdot \right\Vert $ is the
matrix $L_{2}$ norm. Also shown are results for the inverse $\mathbf{dexp}%
^{-1}$ and derivatives, denoted with $\varepsilon ^{\left[ k\right] }$. The
error measures allow assessing the accuracy despite their scale dependence
since there is no bi-invariant metric on $SE(3)$. Notice that there is no
3rd-order approximation of $\left( \mathrm{D}_{\mathbf{X}}\mathbf{dexp}%
\right) \left( \mathbf{U}\right) $, and no 2nd-order approximation of $(%
\mathrm{D}_{\mathbf{X}}^{2}\mathbf{dexp})\left( \mathbf{U}\right) \left( 
\mathbf{Z}\right) $ since the Bernoulli number $B_{3}$ is zero.

An interesting observation is that approximations of the inverse are more
accurate than approximations of $\mathbf{dexp}$ and its derivatives. This is
advantageous for instance when Munthe-Kaas \cite%
{MuntheKaas-BIT1998,IserlesMuntheKaasNrsettZanna2000} or generalized-$\alpha 
$ methods \cite{BrulsCardonaArnold2012} are applied to solve the kinematic
reconstruction equation, which is relevant for Lie group integration of MBS
model, where solutions of the ODE $\dot{\mathbf{X}}=\mathbf{dexp}_{\mathbf{X}%
}^{-1}\mathbf{V}$ on $SE(3)$ are to be computed, and using higher-order
approximations within the integration schemes was shown to be sufficient. 
\begin{figure}[h]
\begin{center}
\includegraphics[draft=false,height=4.5cm]{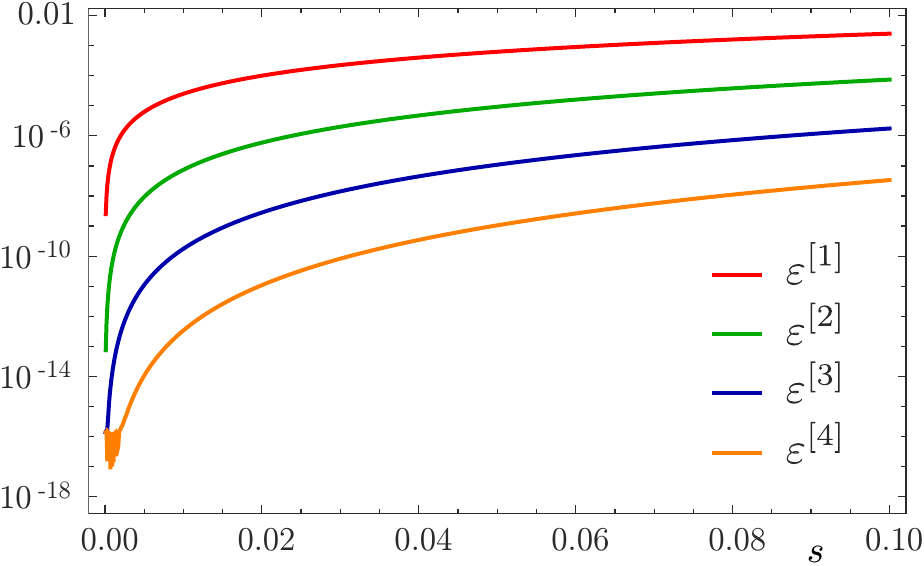}~ %
\includegraphics[draft=false,height=4.5cm]{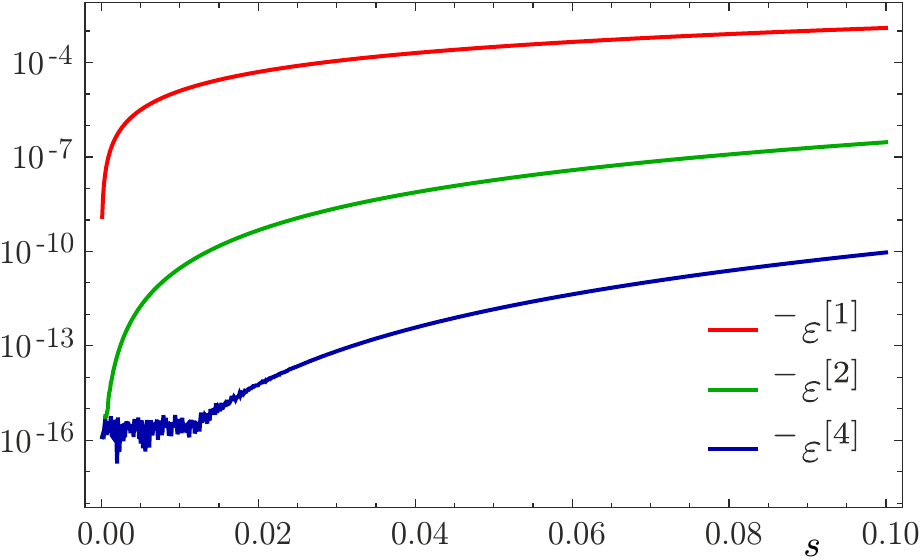} 
\vspace{-2ex}
\end{center}
\caption{Errors $\protect\varepsilon ^{\left[ k\right] }\left( s\right) $ of 
$k$th-order approximation of $\mathbf{dexp}_{\mathbf{X}\left( s\right) }$,
and errors $\protect\varepsilon _{-}^{\left[ k\right] }\left( s\right) $ of $%
\mathbf{dexp}_{\mathbf{X}\left( s\right) }^{-1}$.}
\label{figdexpApprox}
\end{figure}
\begin{figure}[h]
\begin{center}
\includegraphics[draft=false,height=4.5cm]{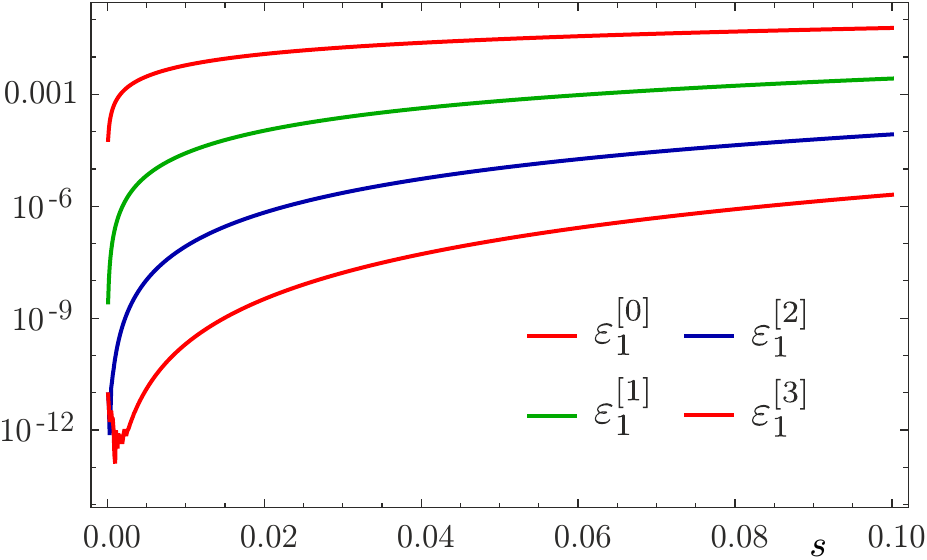}~ %
\includegraphics[draft=false,height=4.5cm]{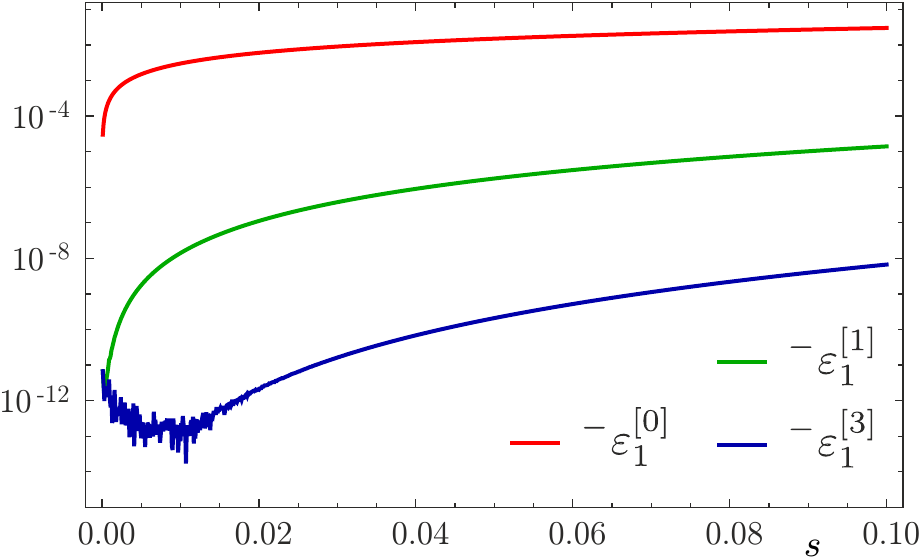} 
\vspace{-2ex}
\end{center}
\caption{Errors $\protect\varepsilon _{1}^{\left[ k\right] }\left( s\right) $
of $k$th-order approximation of $(\mathrm{D}_{\mathbf{X}\left( s\right) }%
\mathbf{dexp})\left( \mathbf{U}\right) $ and its inverse.}
\label{figDdexpApprox013}
\end{figure}
\begin{figure}[h]
\begin{center}
\includegraphics[draft=false,height=4.5cm]{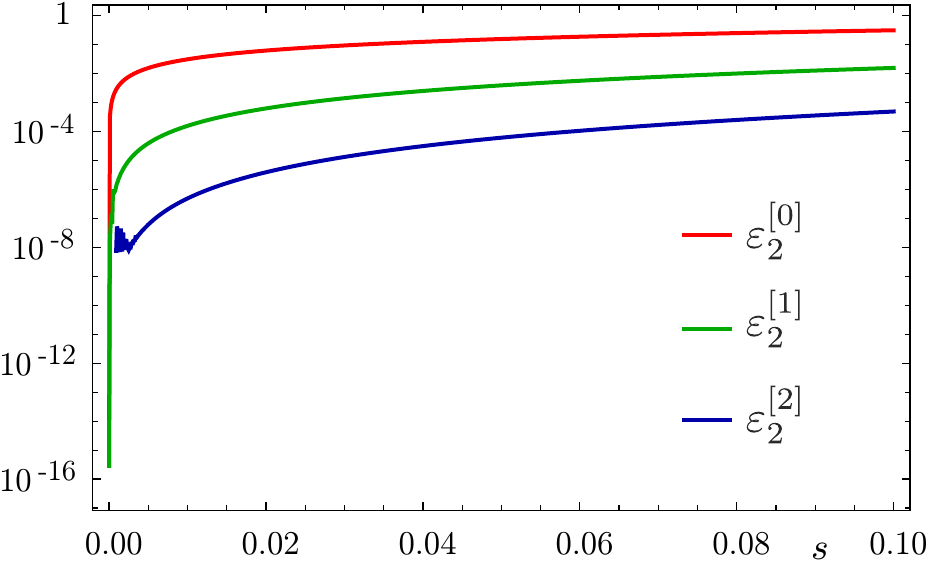}~ %
\includegraphics[draft=false,height=4.5cm]{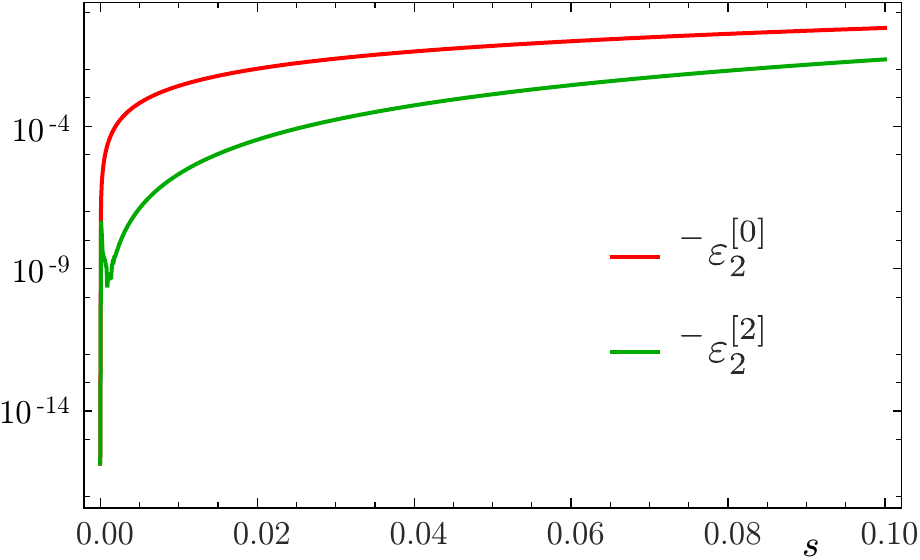} 
\vspace{-2ex}
\end{center}
\caption{Errors $\protect\varepsilon _{2}^{\left[ k\right] }\left( s\right) $
of $k$th-order approximation of $(\mathrm{D}_{\mathbf{X}\left( s\right) }^{2}%
\mathbf{dexp})\left( \mathbf{U}\right) \left( \mathbf{S}\right) $, and its
inverse.}
\label{figD2dexpApprox012}
\end{figure}
\newpage

\section{Application Example: Deformation of a Cosserat Rod}

\label{secCosserat}

The deformation of a straight homogenous rubber rod with a constant
rectangular cross section of $8\times 8\,$mm$^{2}$ and $L=100$\thinspace mm
length is considered. It is modeled as Cosserat-Simo-Reissner rod (see Sec. %
\ref{secApplications}). The deformation field $\boldsymbol{\chi }=\left[ 
\begin{smallmatrix}
\boldsymbol{\kappa } \\ 
\boldsymbol{\rho }%
\end{smallmatrix}%
\right] $ is determined by the kinematic reconstruction equation (\ref%
{RecSE3}), where $\boldsymbol{\kappa }$ describe bending and torsion, and $%
\boldsymbol{\rho }$ shear and compression. In this example, the cross
section displacement, i.e. $\mathbf{X}\left( \tau \right) $, is prescribed.
To obtain a realistic situation, the Rodrigues vector $\mathbf{x}$ and the
vector $\boldsymbol{\rho }$ are prescribed as%
\begin{eqnarray*}
\mathbf{x}\left( \tau \right) &=&\left[ \tfrac{1}{2}\sin \left( 2\pi \tau
\right) ,\tfrac{1}{2}\cos \left( \pi \tau \right) ,\tfrac{1}{2}\sin \left(
2\pi \tau \right) \right] ^{T} \\
\boldsymbol{\rho }\left( \tau \right) &=&\left[ \cos \left( \tfrac{1}{10}%
\sin (2\pi \tau )\right) ,\sin \left( \tfrac{1}{10}\sin (2\pi \tau )\right)
\cos (\sin (2\pi \tau )),\sin \left( \tfrac{1}{10}\sin (2\pi \tau )\right)
\sin (\sin (2\pi \tau ))\right] ^{T}.
\end{eqnarray*}%
The angular deformation $\boldsymbol{\kappa }$ is computed as $\tilde{%
\boldsymbol{\kappa }}=\mathbf{R}^{T}\mathbf{R}^{\prime }$. According to the
definition $\hat{\boldsymbol{\chi }}=\mathbf{H}^{-1}\mathbf{H}^{\prime }$,
the translation $\mathbf{r}\left( \tau \right) $ is found by quadrature of $%
\mathbf{r}^{\prime }=\mathbf{R}\boldsymbol{\rho }$. The reference screw
coordinates $\mathbf{X}_{\mathrm{ref}}\left( \tau \right) $ are then
obtained with the logarithm on $SE(3)$. Fig. \ref{figRodView} shows the
corresponding shape of the rod. From this the exact reference values $%
\boldsymbol{\chi }_{\mathrm{ref}},\boldsymbol{\chi }_{\mathrm{ref}}^{\prime
} $, and $\boldsymbol{\chi }_{\mathrm{ref}}^{\prime \prime }$ are computed.
The actual numerical values presented in the following are not of interest.
The example rather serves to show the applicability of the presented
formulae. 
\begin{figure}[h]
\begin{center}
\includegraphics[draft=false,height=4.6cm]{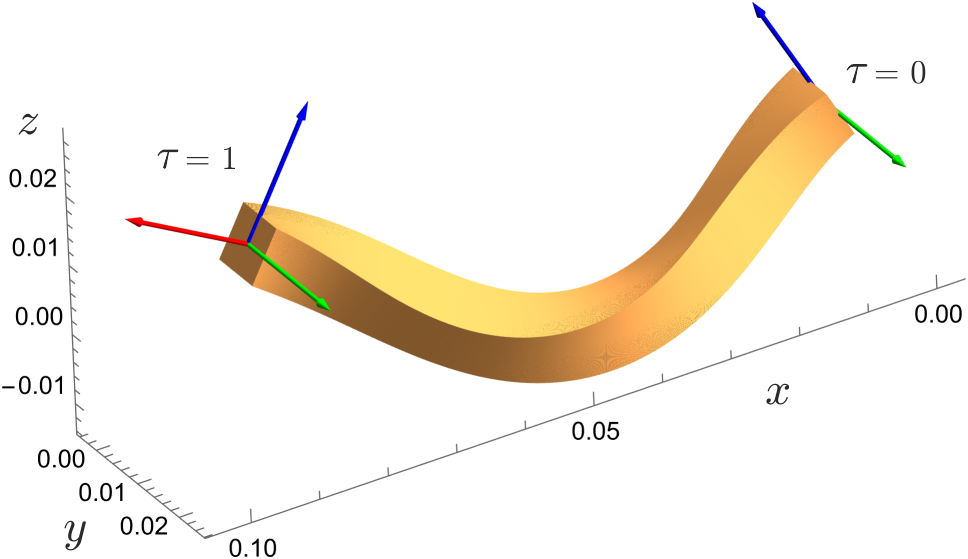}
\end{center}
\caption{Shape of the rubber rod when the displacement of the cross section
frame is according to the prescribed $\mathbf{X}\left( \protect\tau \right) $%
. Shown is the cross section frame at the start and end of the rod.}
\label{figRodView}
\end{figure}

\subsection{Deformation Field}

First the kinematic problem of computing the deformations from the
displacement described by the screw coordinates as function of $\tau =\left[
0,1\right] $ is considered. Note that $\mathbf{x}\left( \tau \right) =%
\mathbf{0}$ at $\tau =0.5$. The first and second derivative of the
deformation field is computed with $\boldsymbol{\chi }^{\prime }=\mathbf{dexp%
}_{-\mathbf{X}}\mathbf{X}^{\prime \prime }-\left( \mathrm{D}_{-\mathbf{X}}%
\mathbf{dexp}\right) (\mathbf{X}^{\prime })\mathbf{X}^{\prime }$ and $%
\boldsymbol{\chi }^{\prime \prime }=\mathbf{dexp}_{-\mathbf{X}}\mathbf{X}%
^{\prime \prime \prime }-2(\mathrm{D}_{\mathbf{X}}\mathbf{dexp})(\mathbf{X}%
^{\prime })\mathbf{X}^{\prime \prime }-(\mathrm{D}_{\mathbf{X}}\mathbf{dexp}%
)(\mathbf{X}^{\prime \prime })\mathbf{X}^{\prime }+(\mathrm{D}_{\mathbf{X}%
}^{2}\mathbf{dexp})(\mathbf{X}^{\prime })(\mathbf{X}^{\prime })\mathbf{X}%
^{\prime }$. Fig. \ref{RodChi} shows the components of $\boldsymbol{\kappa }$
and $\boldsymbol{\rho }$. Clearly visible are the artifacts at $\tau =0.5$,
where $\mathbf{x}=\mathbf{0}$, due to the division by $\left\Vert \mathbf{x}%
\right\Vert =\varphi $ and powers thereof. This issue is readily addressed
by switching to a higher-order approximation of the respective map presented
in Sec. \ref{secApprox1} and \ref{secApprox2} when$\left\Vert \mathbf{x}%
\right\Vert \leq \epsilon $. Using $\epsilon =10^{-5}$ and second-order
local approximations yields numerically smooth results that match the exact
values with computation precision.

The effect of different thresholds for switching between the exact and
approximate relations is demonstrated for computing the first derivative.
Fig. \ref{figDChiApproxDiff} shows the error when the above closed form
expression for $\boldsymbol{\chi }^{\prime }\left( s\right) $ is replaced by
the $k$th-order approximation $\boldsymbol{\chi }^{\prime \left[ k\right] }=%
\mathbf{dexp}_{-\mathbf{X}^{\left[ k\right] }}\mathbf{X}^{\prime \prime %
\left[ k\right] }-\left( \mathrm{D}_{-\mathbf{X}^{\left[ k\right] }}\mathbf{%
dexp}\right) (\mathbf{X}^{\prime \left[ k\right] })\mathbf{X}^{\prime \left[
k\right] }$ if $\left\Vert \mathbf{x}\right\Vert \leq \epsilon $, for the
two rather large threshold values $\epsilon =10^{-2}$ and $\epsilon =10^{-3}$%
. The error is computed as $\varepsilon _{\mathrm{switch}}^{\left[ k\right]
}\left( s\right) :=%
\big\|%
\boldsymbol{\chi }_{\mathrm{ref}}^{\prime }\left( s\right) -\boldsymbol{\chi 
}^{\prime \left[ k\right] }\left( s\right) 
\big\|%
$ if $\left\Vert \mathbf{x}\right\Vert \leq \epsilon $, and $\varepsilon _{%
\mathrm{switch}}^{\left[ k\right] }\left( s\right) :=%
\big\|%
\boldsymbol{\chi }_{\mathrm{ref}}^{\prime }\left( s\right) -\boldsymbol{\chi 
}^{\prime }\left( s\right) 
\big\|%
$ if $\left\Vert \mathbf{x}\right\Vert >\epsilon $. It is apparent from Fig. %
\ref{figDChiApproxDiff} that the $k$th-order approximation near $\mathbf{x}%
\left( 0.5\right) =\mathbf{0}$ yield $k$th-order convergence. For smaller
thresholds $\epsilon $, a second-order approximation suffices to obtain
numerically exact and continuous results. 
\begin{figure}[h]
\begin{center}
\includegraphics[draft=false,height=4.5cm]{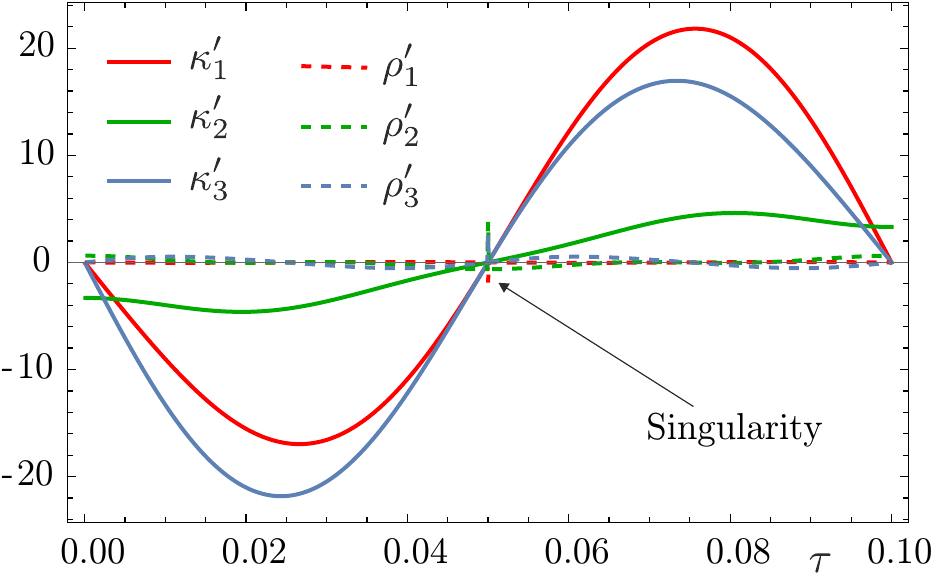}~ %
\includegraphics[draft=false,height=4.5cm]{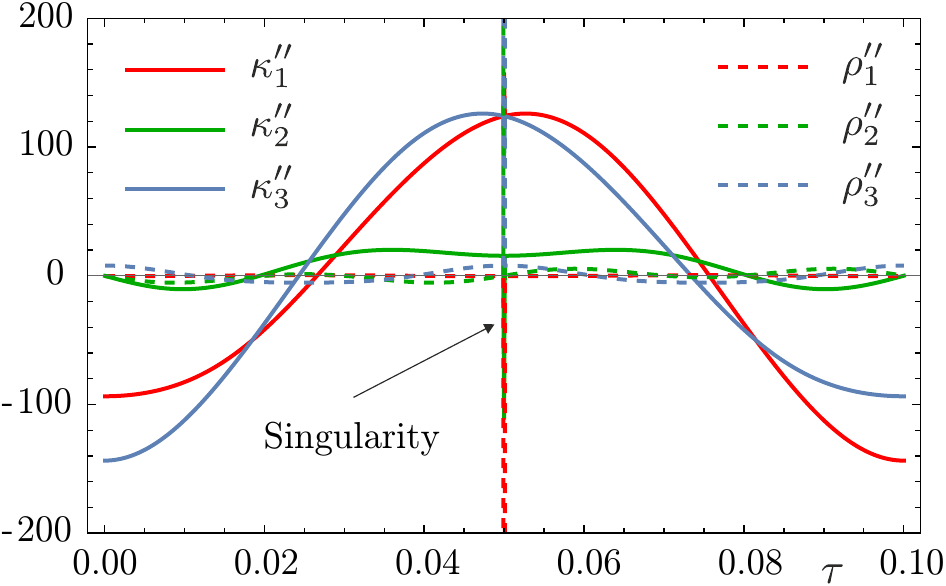} \vspace{%
-2ex}
\end{center}
\caption{First and second derivative of the rod deformation when computed
with (\protect\ref{DdexpSE3}) and (\protect\ref{d2exp}).}
\label{RodChi}
\end{figure}
\begin{figure}[h]
\begin{center}
\includegraphics[draft=false,height=3.9cm]{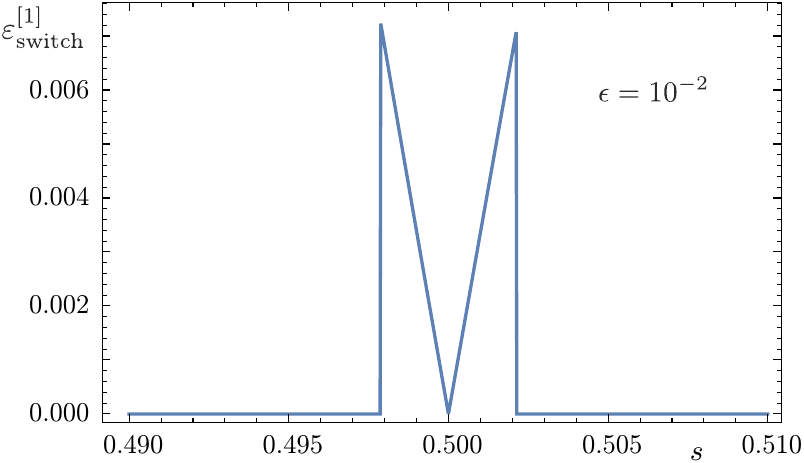}%
~ %
\includegraphics[draft=false,height=3.9cm]{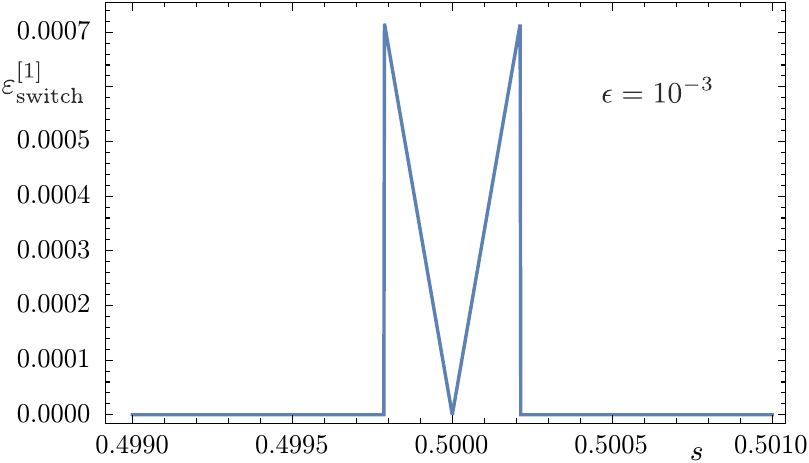}%
\\[0pt]
\includegraphics[draft=false,height=3.9cm]{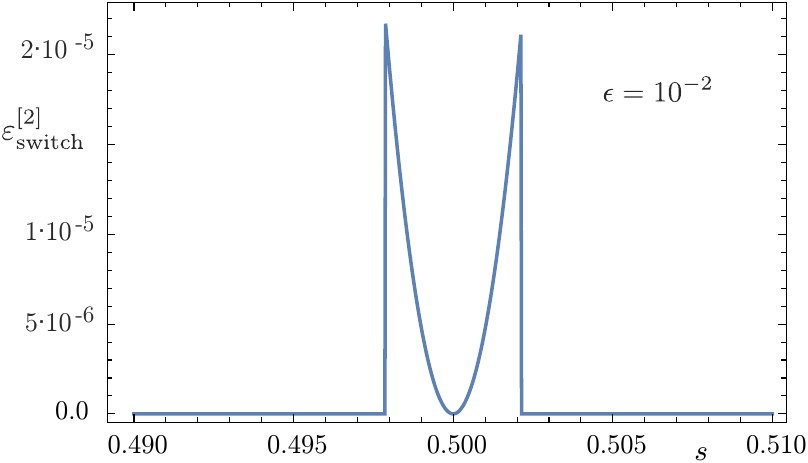}%
~ %
\includegraphics[draft=false,height=3.9cm]{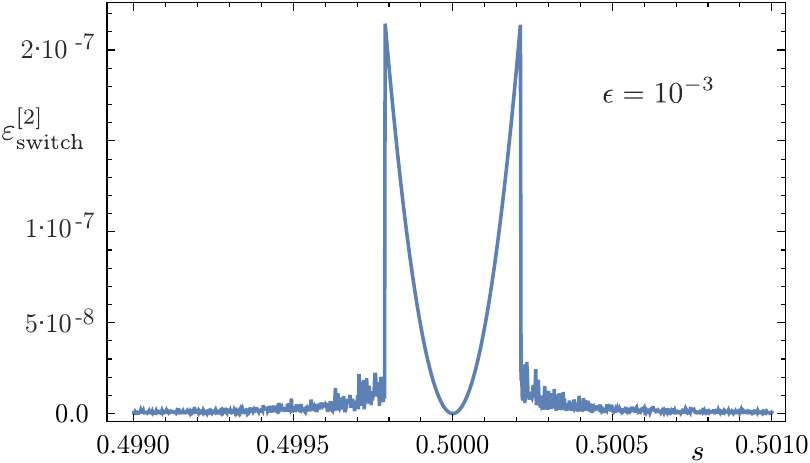}%
\\[0pt]
\includegraphics[draft=false,height=3.9cm]{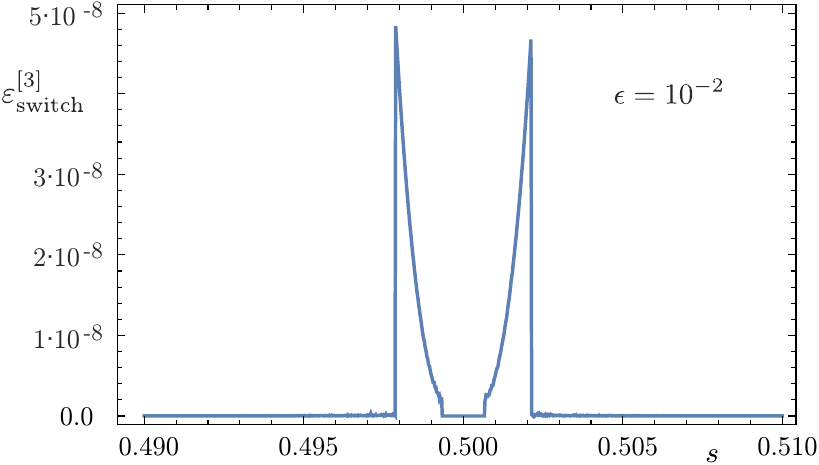}%
~ %
\includegraphics[draft=false,height=3.9cm]{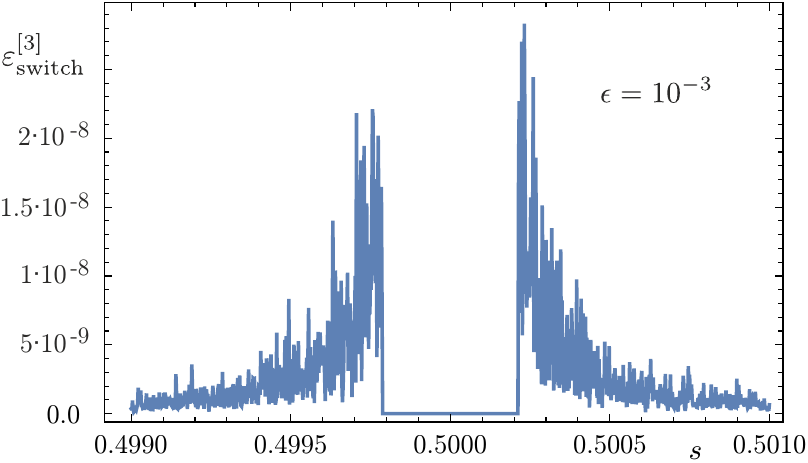}%
\vspace{-2ex}
\end{center}
\caption{Error $\protect\varepsilon _{\mathrm{switch}}^{\left[ k\right]
}\left( s\right) $ when switching between the exact computation of $%
\boldsymbol{\protect\chi }^{\prime }\left( s\right) $ and the $k$th-order
approximations if $\left\Vert \mathbf{x}\right\Vert \leq \protect\epsilon $,
with $\protect\epsilon =10^{-2}$ (left) and $\protect\epsilon =10^{-3}$
(right). Only shown is the relevant range of $s$.}
\label{figDChiApproxDiff}
\end{figure}

\subsection{Gradient and Hessian of Elastic Potential}

The elastic potential of the rod, $\bar{V}=\frac{1}{2}\left( \boldsymbol{%
\chi }-\boldsymbol{\chi }_{0}\right) ^{T}\mathbf{K}\left( \boldsymbol{\chi }-%
\boldsymbol{\chi }_{0}\right) $, is defined in terms of the constant
stiffness matrix $\mathbf{K}=\tfrac{1}{L}\mathrm{diag}\left(
GJ_{x},EI_{yy},EI_{zz},EA,GA,GA\right) $. Therein, $E=10\,$MPa and $G=0.3$%
\thinspace MPa is Young's modulus, $A$ is the cross section area, $%
I_{xx},I_{zz}$ are the second area moments, and $J_{x}$ is the polar area
moment. The gradient $\frac{\partial \bar{V}}{\partial \mathbf{X}}$ and
Hessian of $V$ are computed with (\ref{VGrad}) and (\ref{VHesse}). The
deviation of the gradient from the exact values near $\mathbf{x}=\mathbf{0}$
are shown in Fig. \ref{figRoddexpVGradTabZoom}. Also shown is the matrix
norm $\left\Vert \frac{\partial ^{2}\bar{V}}{\partial \mathbf{X}^{2}}%
\right\Vert $ of the Hessian. These singularities at $\mathbf{x}\left(
0.5\right) =\mathbf{0}$ are avoided by switching to a local approximation as
described in Sec. \ref{secApprox}. When using the 3rd-order approximations
and threshold $\epsilon =10^{-5}$, the numerical results match the exact
solutions up to computation precision. 
\begin{figure}[h]
\begin{center}
\includegraphics[draft=false,height=4.5cm]{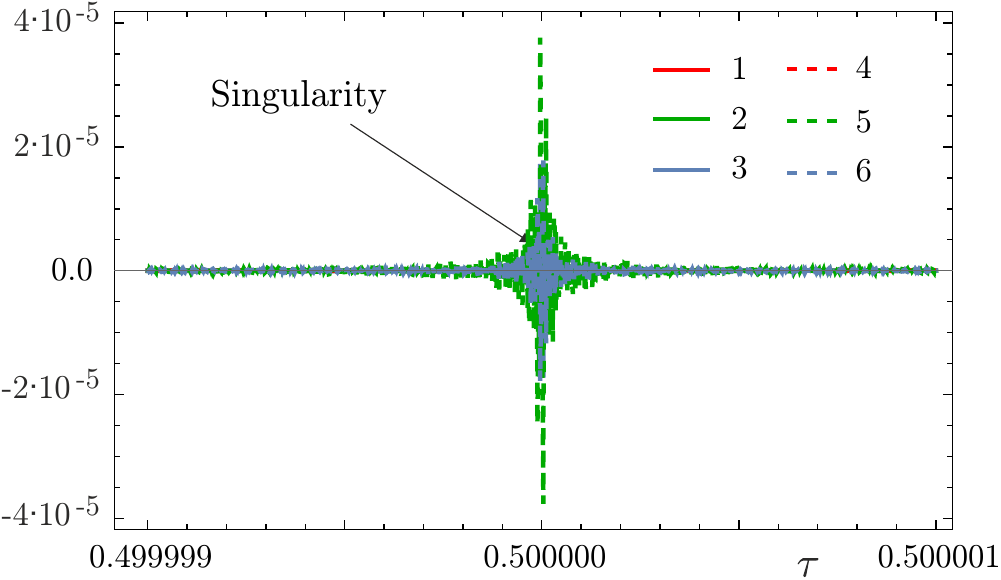}
~~\includegraphics[draft=false,height=4.5cm]{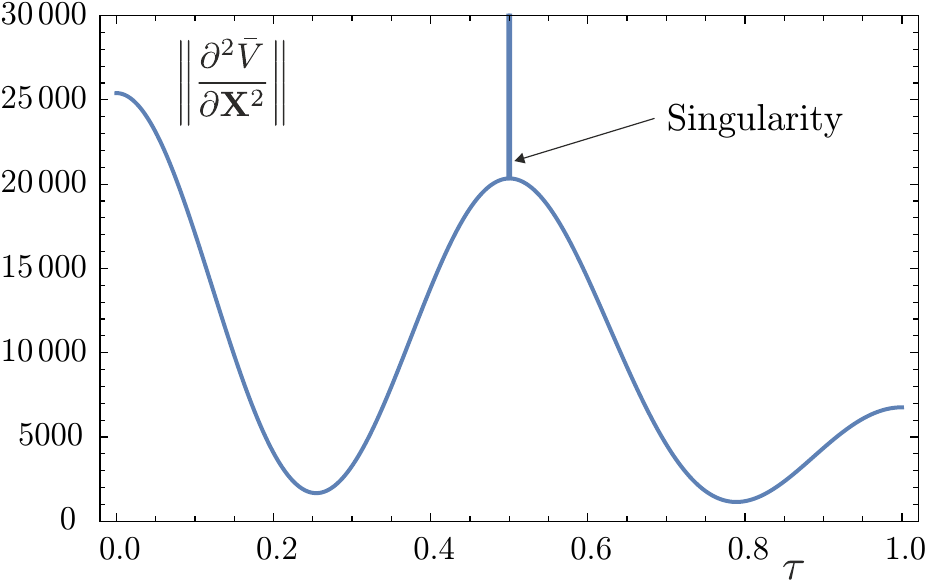} 
\vspace{-2ex}
\end{center}
\caption{Left: Components $i=1,\ldots ,6$ of the error of the gradient $%
\frac{\partial \bar{V}}{\partial \mathbf{X}}$ of the elastic potential
computed with (\protect\ref{VGrad}) using (\protect\ref{PartialDexp}).
Right: $L^{2}$ norm of the Hessian computed with (\protect\ref{VHesse})
using (\protect\ref{H123}).}
\label{figRoddexpVGradTabZoom}
\end{figure}

\section{Conclusion and Outlook}

The right-trivialized differential $\mathbf{dexp}_{\mathbf{X}}$, called the
tangent operator of the exponential map on $SE(3)$ at $\mathbf{X}$, appears
in various contexts of computational mechanics, robotics, and control. Being
a linear map, this tangent map is represented as $6\times 6$ matrix. In many
applications, the first and second directional derivative of the tangent
operator, as well the Jacobian and Hessian of mappings defined in terms of
the tangent operator are required. Traditionally, the latter is represented
in a block partitioned form that stems from the semi-direct product
structure of $SE(3)$. Consequently, the first directional derivative
reported in the literature involves second directional derivative of the
right-trivialized differential of the exponential map on $SO(3)$. This fact,
and the fact that each block entry is treated separately leads to convoluted
formulae and to computational issues at the (removable) singularity at $%
\mathbf{X}=\mathbf{0}$. The second directional derivative of $\mathbf{dexp}_{%
\mathbf{X}}$, or the Jacobian and Hessian of functions defined with it, have
not yet been reported. In this paper, the first and second derivatives, as
well the relevant Jacobian and Hessian were derived in closed form avoiding
the block partitioning. Additionally, higher-order local approximations were
derived for all closed form relations. These two results combined admit
numerically robust evaluation of relations needed in most applications. 
\color[rgb]{0,0,0}%
The switching between the respective exact and approximate relations was not
addressed in this paper. Switching methods, preferably with well-defined
accuracy, will be subject of further research. From a practical point of
view, while the results of the block partitioned and the introduced
formulation are indeed identical, the $6\times 6$ matrix representation is
easier to implement and involves fewer objects to manipulate. This shall be
obvious comparing (\ref{dexpSE3}) and (\ref{diffDexpSE3}), for instance.%
\color{black}%

For applications, where the canonical coordinates remain small, e.g. within
finite element methods or in flexible MBS dynamics, the exponential map may
be replaced by the Cayley map (or a higher-order Cayley map) on $SE(3)$. For
these alternative coordinate maps, the corresponding derivatives will be
derived in future research. They can be expected to be simpler than those
presented, which comes at the expense that the Cayley map is a local
approximation of the exponential map.

\color[rgb]{0,0,0}%

\section*{Funding}

This work has been supported by the "LCM-K2 Center for Symbiotic
Mechatronics" within the framework of the Austrian COMET-K2 program.

\section*{Data Availability Statement}

A Mathematica$^{\copyright }$ library comprising all presented formulae is
openly available at Github:
https://github.com/andreasmuellerjku/SO3-SE3-exp-Derivatives-Jacobian-Hessian%
\color{black}%

\section*{Conflicts of Interest}
The author declares that there is no conflict of interest regarding the publication of this work.

\appendix

\section{Exponential map and derivatives using block partitioning according
to the semidirect product}

\label{secAppSE3Block}

In the following, the closed form relations of the exp map and its
derivatives using the $3\times 3$ block partitioning, according to the
semi-direct product structure $SE\left( 3\right) =SO\left( 3\right) \ltimes {%
\mathbb{R}}^{3}$, are summarized. Their derivations can be found in \cite%
{RSPA2021}. In the following, the screw coordinate vector is $\mathbf{X}%
=\left( \mathbf{x},\mathbf{y}\right) \in {\mathbb{R}}^{3}\times {\mathbb{R}}%
^{3}$. The exponential map (\ref{expSE3}) can be written with the the
rotation matrix $\mathbf{R}=\exp \tilde{\mathbf{x}}$ as%
\begin{equation}
\exp (\hat{\mathbf{X}})=%
\begin{bmatrix}
\mathbf{R} & \ \tfrac{1}{\left\Vert \mathbf{x}\right\Vert ^{2}}(\mathbf{I}-%
\mathbf{R})\tilde{\mathbf{x}}\mathbf{y}+h\mathbf{y} \\ 
\mathbf{0} & 1%
\end{bmatrix}%
,\ \mathrm{for\ }\mathbf{x}\neq \mathbf{0}  \label{SE3exp1}
\end{equation}%
in which $h:=\mathbf{x}^{T}\mathbf{y}/\left\Vert \mathbf{x}\right\Vert ^{2}$
is the instantaneous pitch of the screw motion. The pitch for a pure
rotation is $h=0$. A pure translation corresponds to the limit $h=\infty $,
to which $\mathbf{X}=\left( \mathbf{0},\mathbf{y}\right) $ is assigned, and (%
\ref{SE3exp1}) becomes $\exp (\hat{\mathbf{X}})=\left[ 
\begin{smallmatrix}
\mathbf{I} & \mathbf{y} \\ 
\mathbf{0} & 1%
\end{smallmatrix}%
\right] $.

The block partitioned matrix form of the right-trivialized differential (\ref%
{dexpSE3Block}) on $SE\left( 3\right) $ involves the differential $\mathbf{%
dexp}_{\mathbf{x}}$ on $SO\left( 3\right) $ in (\ref{dexpSO31}), but also
its directional derivative along $\mathbf{y}\in {\mathbb{R}}^{3}$. The
latter admits the explicit expression%
\begin{eqnarray}
\left( \mathrm{D}_{\mathbf{x}}\mathbf{dexp}\right) 
\hspace{-0.5ex}%
\left( \mathbf{y}\right) &=&\tfrac{\beta }{2}\tilde{\mathbf{y}}+\delta
\left( \tilde{\mathbf{x}}\tilde{\mathbf{y}}+\tilde{\mathbf{y}}\tilde{\mathbf{%
x}}\right) +\tfrac{\mathbf{x}^{T}\mathbf{y}}{\left\Vert \mathbf{x}%
\right\Vert ^{2}}\left( \alpha -\beta \right) \tilde{\mathbf{x}}+\tfrac{%
\mathbf{x}^{T}\mathbf{y}}{\left\Vert \mathbf{x}\right\Vert ^{2}}\left( 
\tfrac{\beta }{2}-3\delta \right) \tilde{\mathbf{x}}^{2}
\label{diffDexpSO33} \\
(\mathrm{D}_{\mathbf{x}}\mathbf{dexp}^{-1})%
\hspace{-0.5ex}%
\left( \mathbf{y}\right) &=&-\tfrac{1}{2}\tilde{\mathbf{y}}+\frac{1}{%
\left\Vert \mathbf{x}\right\Vert ^{2}}\left( 1-\gamma \right) \left( \tilde{%
\mathbf{x}}\tilde{\mathbf{y}}+\tilde{\mathbf{y}}\tilde{\mathbf{x}}\right) +%
\tfrac{\mathbf{x}^{T}\mathbf{y}}{\left\Vert \mathbf{x}\right\Vert ^{4}}%
\left( \tfrac{1}{\beta }+\gamma -2\right) \tilde{\mathbf{x}}^{2}.
\label{diffDexpSO34}
\end{eqnarray}%
When $\mathbf{x}\rightarrow \mathbf{0}$, the limits of (\ref{diffDexpSO33})
and (\ref{diffDexpSO34}) are%
\begin{equation}
\left( \mathrm{D}_{\mathbf{0}}\mathbf{dexp}\right) 
\hspace{-0.5ex}%
\left( \mathbf{y}\right) =\tfrac{1}{2}\tilde{\mathbf{y}},\ \ (\mathrm{D}_{%
\mathbf{0}}\mathbf{Ddexp}^{-1})%
\hspace{-0.5ex}%
\left( \mathbf{U}\right) =-\tfrac{1}{2}\tilde{\mathbf{y}}.
\end{equation}%
The directional derivative of the differential $\mathbf{dexp}_{\mathbf{X}}$
on $SE\left( 3\right) $ in (\ref{dexpSE3Block}), and its inverse, along $%
\mathbf{U}=\left( \mathbf{u},\mathbf{v}\right) $ involves the second
derivative of $\mathbf{dexp}_{\mathbf{x}}$ on $SO\left( 3\right) $. Denote
the directional derivative (\ref{diffDexpSO33}) of the matrix $\mathbf{dexp}%
_{\mathbf{x}}$ with $\mathbf{Ddexp}\left( \mathbf{X}\right) :=(\mathrm{D}_{%
\mathbf{x}}\mathbf{dexp})%
\hspace{-0.5ex}%
\left( \mathbf{y}\right) $. Then (\ref{dexpSE3Block}) leads to%
\begin{eqnarray}
(\mathrm{D}_{\mathbf{X}}\mathbf{dexp})%
\hspace{-0.5ex}%
\left( \mathbf{U}\right) &=&%
\begin{bmatrix}
\mathrm{D}_{\mathbf{x}}\mathbf{dexp})%
\hspace{-0.5ex}%
\left( \mathbf{u}\right) & \mathbf{0} \\ 
(\mathrm{D}_{\mathbf{X}}\mathbf{Ddexp})%
\hspace{-0.5ex}%
\left( \mathbf{U}\right) & (\mathrm{D}_{\mathbf{x}}\mathbf{dexp})%
\hspace{-0.5ex}%
\left( \mathbf{u}\right)%
\end{bmatrix}
\label{diffDexpSE3} \\
(\mathrm{D}_{\mathbf{X}}\mathbf{dexp}^{-1})%
\hspace{-0.5ex}%
\left( \mathbf{U}\right) &=&%
\begin{bmatrix}
(\mathrm{D}_{\mathbf{x}}\mathbf{dexp}^{-1})%
\hspace{-0.5ex}%
\left( \mathbf{u}\right) & \mathbf{0} \\ 
(\mathrm{D}_{\mathbf{X}}\mathbf{Ddexp}^{-1})%
\hspace{-0.5ex}%
\left( \mathbf{U}\right) & (\mathrm{D}_{\mathbf{x}}\mathbf{dexp}^{-1})%
\hspace{-0.5ex}%
\left( \mathbf{u}\right)%
\end{bmatrix}%
.
\end{eqnarray}%
Therein $(\mathrm{D}_{\mathbf{X}}\mathbf{Ddexp})%
\hspace{-0.5ex}%
\left( \mathbf{U}\right) $ is the directional derivative of the matrix $%
\mathbf{Ddexp}\left( \mathbf{X}\right) $ along $\mathbf{U}=\left( \mathbf{u},%
\mathbf{v}\right) $. It can be expressed in closed form, and for the
inverse, as%
\begin{align}
(\mathrm{D}_{\mathbf{X}}& \mathbf{Ddexp})%
\hspace{-0.5ex}%
\left( \mathbf{U}\right) =\tfrac{\beta }{2}\tilde{\mathbf{v}}+\tfrac{\alpha
-\beta }{\left\Vert \mathbf{x}\right\Vert ^{2}}\left( (\mathbf{x}^{T}\mathbf{%
u})\tilde{\mathbf{y}}+(\mathbf{x}^{T}\mathbf{v+y}^{T}\mathbf{u})\tilde{%
\mathbf{x}}+(\mathbf{x}^{T}\mathbf{y})\tilde{\mathbf{u}}\right)  \notag \\
& -\tfrac{\beta }{2\left\Vert \mathbf{x}\right\Vert ^{2}}(\mathbf{x}^{T}%
\mathbf{u})(\mathbf{x}^{T}\mathbf{y})\tilde{\mathbf{x}}+\delta \left( \tilde{%
\mathbf{x}}\tilde{\mathbf{v}}+\tilde{\mathbf{v}}\tilde{\mathbf{x}}+\tilde{%
\mathbf{y}}\tilde{\mathbf{u}}+\tilde{\mathbf{u}}\tilde{\mathbf{y}}\right) 
\notag \\
& +\tfrac{\beta /2-3\delta }{\left\Vert \mathbf{x}\right\Vert ^{2}}\left( (%
\mathbf{x}^{T}\mathbf{y})(\tilde{\mathbf{x}}\tilde{\mathbf{u}}+\tilde{%
\mathbf{u}}\tilde{\mathbf{x}})+(\mathbf{x}^{T}\mathbf{u})(\tilde{\mathbf{x}}%
\tilde{\mathbf{y}}+\tilde{\mathbf{y}}\tilde{\mathbf{x}})+(\mathbf{x}^{T}%
\mathbf{v}+\mathbf{y}^{T}\mathbf{u})\tilde{\mathbf{x}}^{2}\right)  \notag \\
& +\tfrac{1}{\left\Vert \mathbf{x}\right\Vert ^{4}}(\mathbf{x}^{T}\mathbf{y}%
)(\mathbf{x}^{T}\mathbf{u})\left( (1-5\alpha +4\beta )\tilde{\mathbf{x}}%
+(\alpha -\tfrac{7}{2}\beta +15\delta )\tilde{\mathbf{x}}^{2}\right)  \notag
\\
(\mathrm{D}_{\mathbf{X}}& \mathbf{Ddexp}^{-1})%
\hspace{-0.5ex}%
\left( \mathbf{U}\right) =-\tfrac{1}{2}\tilde{\mathbf{v}}+\tfrac{1-\gamma }{%
\left\Vert \mathbf{x}\right\Vert ^{2}}\left( \tilde{\mathbf{x}}\tilde{%
\mathbf{v}}+\tilde{\mathbf{v}}\tilde{\mathbf{x}}+\tilde{\mathbf{y}}\tilde{%
\mathbf{u}}+\tilde{\mathbf{u}}\tilde{\mathbf{y}}\right) +\tfrac{1}{4}(%
\mathbf{x}^{T}\mathbf{u})(\tilde{\mathbf{x}}\tilde{\mathbf{y}}+\tilde{%
\mathbf{y}}\tilde{\mathbf{x}})  \notag \\
& -\tfrac{1}{\left\Vert \mathbf{x}\right\Vert ^{4}}%
\Big%
(\left( 1-\gamma \right) (\mathbf{x}^{T}\mathbf{u})(2+\gamma )(\tilde{%
\mathbf{x}}\tilde{\mathbf{y}}+\tilde{\mathbf{y}}\tilde{\mathbf{x}})  \notag
\\
& +\left( (\gamma +\tfrac{1}{\beta }-2)\left( (\mathbf{x}^{T}\mathbf{y})(%
\tilde{\mathbf{x}}\tilde{\mathbf{u}}+\tilde{\mathbf{u}}\tilde{\mathbf{x}})+(%
\mathbf{x}^{T}\mathbf{v}+\mathbf{y}^{T}\mathbf{u})\right) -\tfrac{1}{4}(%
\mathbf{x}^{T}\mathbf{y})(\mathbf{x}^{T}\mathbf{u})\right) \tilde{\mathbf{x}}%
^{2}%
\Big%
)  \notag \\
& +\tfrac{1}{\left\Vert \mathbf{x}\right\Vert ^{6}}(\mathbf{x}^{T}\mathbf{y}%
)(\mathbf{x}^{T}\mathbf{u})\left( 8-3\gamma -\gamma ^{2}-\tfrac{2}{\beta ^{2}%
}\left( \alpha +\beta \right) \right) \tilde{\mathbf{x}}^{2}.
\label{diff2dexpSO31}
\end{align}%
The limits for $\mathbf{X}\rightarrow \mathbf{0}$ are%
\begin{equation}
(\mathrm{D}_{\mathbf{0}}\mathbf{Ddexp})%
\hspace{-0.5ex}%
\left( \mathbf{U}\right) =\tfrac{1}{2}\tilde{\mathbf{v}}+\tfrac{1}{6}(\tilde{%
\mathbf{y}}\tilde{\mathbf{u}}+\tilde{\mathbf{u}}\tilde{\mathbf{y}}),\ (%
\mathrm{D}_{\mathbf{0}}\mathbf{Ddexp}^{-1})%
\hspace{-0.5ex}%
\left( \mathbf{U}\right) =-\tfrac{1}{2}\tilde{\mathbf{v}}+\tfrac{1}{12}(%
\tilde{\mathbf{y}}\tilde{\mathbf{u}}+\tilde{\mathbf{u}}\tilde{\mathbf{y}}).
\end{equation}

\section{%
\color[rgb]{0,0,0}%
Nomenclature}

\color[rgb]{0,0,0}%
\begin{tabular}{ll}
$\mathbf{x}\in \mathbb{R}^{3}$ & Vector of canonical coordinates on $SO(3)$
\\ 
$\mathbf{X}\in \mathbb{R}^{6}$ & Vector of canonical coordinates on $SE(3)$
\\ 
$\tilde{\mathbf{x}}\in so\left( 3\right) $ & Skew symmetric matrix
associated to vector $\mathbf{x}\in \mathbb{R}^{3}$ \\ 
$\hat{\mathbf{X}}\in se\left( 3\right) $ & $se\left( 3\right) $ matrix (\ref%
{Xhat}) associated to vector $\mathbf{X}\in \mathbb{R}^{6}$ \\ 
$\exp \mathbf{x},\exp \tilde{\mathbf{x}}$ & exponential map on $SO(3)$ \\ 
$\exp \mathbf{X},\exp \hat{\mathbf{X}}$ & exponential map on $SE(3)$ \\ 
$\mathbf{dexp}_{\mathbf{x}},\mathbf{dexp}_{\tilde{\mathbf{x}}}$ & Matrix
representation of the right-trivialized differential on $SO\left( 3\right) $
\\ 
& $\mathbf{x}\in \mathbb{R}^{3}$ is the vector of canonical coordinates \\ 
$\mathbf{dexp}_{\mathbf{X}},\mathbf{dexp}_{\hat{\mathbf{X}}}$ & Matrix
representation of the right-trivialized differential on $SE\left( 3\right) $
\\ 
& $\mathbf{X}\in \mathbb{R}^{6}$ is the vector of canonical coordinates \\ 
$\left( \mathrm{D}_{\mathbf{X}}\mathbf{dexp}\right) 
\hspace{-0.5ex}%
\left( \mathbf{Y}\right) $ & Directional derivative of $\mathbf{dexp}$ at $%
\mathbf{X}\in \mathbb{R}^{6}\cong se\left( 3\right) $ along $\mathbf{Y}\in 
\mathbb{R}^{6}$ \\ 
$\left( \mathrm{D}_{\mathbf{X}}^{2}\mathbf{dexp}\right) 
\hspace{-0.5ex}%
\left( \mathbf{Y}\right) \left( \mathbf{U}\right) $ & Second directional
derivative of $\mathbf{dexp}$ at $\mathbf{X}\in \mathbb{R}^{6}\cong se\left(
3\right) $ along $\mathbf{Y},\mathbf{U}\in \mathbb{R}^{6}$ \\ 
$\mathbf{V}=\left[ 
\begin{smallmatrix}
\boldsymbol{\omega } \\ 
\mathbf{v}%
\end{smallmatrix}%
\right] \in \mathbb{R}^{6}\cong se\left( 3\right) $ & Twist of a frame /
rigid body \\ 
$\boldsymbol{\chi }=\left[ 
\begin{smallmatrix}
\boldsymbol{\kappa } \\ 
\boldsymbol{\rho }%
\end{smallmatrix}%
\right] \in \mathbb{R}^{6}\cong se\left( 3\right) $ & Deformation measure of
a Cosserat rod \\ 
$\mathbf{ad}_{\mathbf{X}}=\left[ 
\begin{smallmatrix}
\tilde{\mathbf{x}} & \mathbf{0} \\ 
\tilde{\mathbf{y}} & \tilde{\mathbf{x}}%
\end{smallmatrix}%
\right] $ & Matrix representation of adjoint operator on $SE(3)$. Then $%
\left[ \mathbf{X},\mathbf{Y}\right] =\mathbf{ad}_{\mathbf{X}}\mathbf{Y}$ is
the Lie bracket on $se\left( 3\right) $ in vector representation. \\ 
$\overline{\mathbf{ad}}_{\mathbf{U}}:=%
\begin{bmatrix}
\tilde{\mathbf{u}} & \tilde{\mathbf{v}} \\ 
\tilde{\mathbf{v}} & \mathbf{0}%
\end{bmatrix}%
$ & Skew symmetric matrix so that $\mathbf{ad}_{\mathbf{X}}^{T}\mathbf{U}=%
\overline{\mathbf{ad}}_{\mathbf{U}}\mathbf{X}$ \\ 
$f\left( \mathbf{X}\right) =\mathbf{dexp}_{\mathbf{X}}\mathbf{Z}$ & 
Evaluation map of $\mathbf{dexp}$ at $\mathbf{X}$ - $\mathbf{dexp}_{\mathbf{X%
}}$ evaluated with $\mathbf{Z}$ \\ 
$\bar{f}\left( \mathbf{X}\right) =\mathbf{dexp}_{\mathbf{X}}^{T}\mathbf{Z}$
& Evaluation map of $\mathbf{dexp}_{\mathbf{X}}^{T}$ at $\mathbf{X}$ - $%
\mathbf{dexp}_{\mathbf{X}}^{T}$ evaluated with $\mathbf{Z}$ \\ 
$h\left( \mathbf{X}\right) =\mathbf{Q}^{T}\mathbf{dexp}_{\mathbf{X}}\mathbf{Z%
}$ & Bilinear map obtained as pairing of evaluation map of $\mathbf{dexp}_{%
\mathbf{X}}\mathbf{Y}$ with $\mathbf{Q}\in \mathbb{R}^{6}\cong se^{\ast
}\left( 3\right) $%
\end{tabular}%
\newline

\color{black}%
\bibliographystyle{IEEEtran}
\bibliography{dexpDerivatives}

\end{document}